\documentclass[a4paper,11pt]{article}
\usepackage[a4paper,margin=2.5cm]{geometry}
\usepackage[utf8]{inputenc}
\usepackage[english]{babel}
\usepackage{lmodern} 
\usepackage{cite}
\usepackage{tikz}
\usepackage{fancyhdr}
\usepackage{newlfont}
\usepackage{indentfirst}
\usepackage{amssymb}
\usepackage{comment}
\usepackage{amsmath,esint}
\usepackage{bm}
\usepackage{bbm}
\usepackage[]{hyperref}
\usepackage{braket}
\usepackage{latexsym}
\usepackage{mathrsfs}
\usepackage{mathtools}
\usepackage{xcolor}
\usepackage{physics}
\usepackage{amsthm}
\usepackage{textcomp}
\usepackage[capitalize,nameinlink,noabbrev]{cleveref}
\newtheorem{defm}{Definition}[section]
\newtheorem{thm}{Theorem}[section]

\newtheorem{lem}{Lemma}[section]
\newtheorem{prop}{Proposition}[section]
\newtheorem{assumption}{Assumption}[section]
\newtheorem{rem}{Remark}

\theoremstyle{plain}

\newcommand{\nab}{\nabla}

\newcommand{\bl}{\biggl}
\newcommand{\br}{\biggr}
\newcommand{\rom}[1]{\mathrm{#1}}
\newcommand{\bigslant}[2]{{\raisebox{.2em}{$#1$}\left/\raisebox{-.2em}{$#2$}\right.}}
\newcommand{\eqname}[1]{\tag*{(#1)}}

\newcommand{\vertiii}[1]{{\left\vert\kern-0.25ex\left\vert\kern-0.25ex\left\vert #1 
    \right\vert\kern-0.25ex\right\vert\kern-0.25ex\right\vert}}

\newcommand{\bbR}{\mathbb{R}}
\newcommand{\bmL}{\bm{L}}
\newcommand{\bmH}{\bm{H}}
\newcommand{\bmD}{\bm{D}}
\newcommand{\bbL}{\mathbb{L}}
\newcommand{\bbN}{\mathbb{N}}

\newcommand{\bbH}{\mathbb{H}}
\newcommand{\bbW}{\mathbb{W}}
\newcommand{\bbQS}{\mathbb{QS}}
\newcommand{\bbD}{\mathbb{D}}
\newcommand{\bbS}{\mathbb{S}}

\newcommand{\tg}{\tilde{g}}
\newcommand{\hg}{\hat{g}}

\newcommand\restr[2]{{
  \left.\kern-\nulldelimiterspace
  #1
  \vphantom{\big|}
  \right|_{#2} 
  }}
  \newcommand{\tobias}[1]{{\color{magenta} TK: #1}}

 \title{Well-posedness and long time dynamics for a quasi-geostrophic ocean-atmosphere model with radiation balance}
  \author{Federico Fornasaro$^1$, Tobias Kuna$^2$, Giulia Carigi$^2$}
  \date{}
  
\begin{document}
\maketitle

\begin{abstract}
    %
    We investigate a coupled atmosphere-ocean model including the mechanical and thermodynamical interaction between the two fluids for the mid-latitudes. The formulation combines a multilayer quasi-geostrophic dynamical framework with temperature equations incorporating long- and short-wave radiative forcing, as in energy balance models. Within a suitable functional framework, we establish the existence and uniqueness of solutions, and their continuous dependence on the radiation parameters. We also prove that the long-time dynamics are described by a finite-dimensional global attractor and, moreover, that the system possesses a finite set of determining modes that governs its asymptotic behaviour. In particular, we show that the long-term evolution of the ocean’s temperature can be reconstructed solely from observations of the velocity fields across the model’s layers.
\end{abstract}
\vspace{1em}
\textit{Secondary}: 35B41, 
86A08, 76U60 


\section{Introduction}

Research in weather and climate science is a topic of ever increased popularity in recent decades. It encompasses  understanding the effect and sustainability of human activities in relation to the Earth system, the impact on the prospects of human life, mitigating the danger for the most vulnerable ecosystems, improving  prediction of extreme events, identifying tipping-points, and enabling more effective countermeasures.

As pointed out in \cite{GhilLucVariability}, ``the climate system is a forced, dissipative, chaotic system that is out of equilibrium and whose complex natural variability arises from the interplay of positive and negative feedbacks, instabilities and saturation mechanisms. These processes span a broad range of spatial and temporal scales and include many chemical species and all physical phases [...]; it evolves, furthermore, under the action of large-scale agents that drive and modulate its evolution, mainly differential solar heating and the Earth’s rotation and gravitation.'' 

Beside all the aforedescribed features, just the mere beauty and intricateness of climate models is more than enough to motivate substantial efforts in studying them from a mathematical perspective, which may lead also to transferable ideas and techniques to the study of other systems.

In this paper we study the well-posedness of a coupled atmosphere-ocean model suggested by \cite{VannitsemGhil} and its behaviour for long times in terms of global attractor and determining modes. The model proved effective in both medium and long time analysis from a numerical perspective (see \cite{VannitsemGhil},\cite{MAOOAM},\cite{VanLuc},\cite{HamVan}) representing both mechanical and nonlinear thermodynamic features of the physical system via two coupled sets of equations.
The systems contains at the same time nonlinearities of transport and reaction-diffusion type coming from an EBM type models. 
The model presents novelty with respect to the current literature as it considers both atmosphere and ocean and full temperature evolutions for them. To the best of our knowledge, other similar works do not contain all of these features at once. See \cref{sec:lit} for more details.


To better understand the nature and relevance of the model studied in this paper we should recall the concept of hierarchy of models in geophysics. The most comprehensive model, incorporating all the facets of the climate system may not always be the best tool to investigate some of its aspects. Simplified models may, like a caricature, show key features of the system more clearly, lead to a better qualitative understanding, or allow, due to smaller computational cost, a better exploration of the variability, extremes and other rarer states of the system. 
One family of such simplified models is that of Energy Balance models (EBM), see for example \cite{North}. These models concentrate on the interplay of temperature and radiation by replacing fluid motion by diffusive terms leading to systems of reaction-diffusion equations on the sphere, called 2D (EBM), or assuming longitudinal rotation invariance leading to reaction-diffusion models on an interval, which then are 1D-models. Totally neglecting any spatial extension of the earth, the so-called 0D-models, leads to systems of coupled ODEs.

From a dynamical perspective, models of intermediate complexity include aspects of fluid motions but in a simplified form with respect to the full complexity of compressible 3D Navier-Stokes equations. Relevant examples are the primitive equations and the quasi-geostrophic (QG) equations. The former are based on hydrostatic balance, which captures phenomena on synoptic scales, namely the space-time scales where weather patterns appear. The latter are obtained as a perturbation around the geostrophic balance, namely when Coriolis forces dominate the flow, an assumption valid at the mid-latitudes. 

The hierarchy is not totally ordered as models may or may not include other components of the climate system, as for example ocean, infrared radiation, moisture, ice sheets etc.
Though real simplifications, the models in the hierarchy may capture quite complex features. For example, numerical weather prediction started in the late forties with the use of layered QG models for the atmosphere, see \cite{CharneyNeumann}, used up to the seventies. Quasi-geostrophic multi-layer models are still used in recent works in order to make fast and high-resolution simulations for rotating fluids, as in \cite{Quag}, or to describe the greenhouse gas concentration and carbon cycle in paleoclimate, see \cite{Holocene}.


In this paper, we will consider a model consisting out of a two-layer QG atmosphere coupled to temperature fields of the atmosphere transported with the flow and a one-layer QG ocean coupled to an inactive deep ocean and an associated temperature field, sometimes referred to as one-and-a-half layer QG. The full model is based on the idea of treating as dominant the gravitation, Earth's rotation, and solar radiation and keeping the dominant circulation of atmosphere and ocean. The leading role of rotation restricts the validity of the model to mid-latitudes and synoptic scales. The model is suggested in \cite{MAOOAM} though with a different boundary condition for the streamfunctions, following the authors we will use the acronym (MAOOAM) for this model. The dynamic model is arguably the simplest model that captures important key features of weather and climate, like the baroclinic instabilities, the interaction of atmosphere and ocean, the radiation balance including changes of albedo.
MAOOAM has been proposed to study for example the low frequency variability of the atmosphere as in \cite{VannitsemGhil} or the multistability of the ocean-atmosphere system as in \cite{VanLuc}, \cite{HamVan}, or assess Data-Assimilation (DA) algorithms as in \cite{Penny} and hybrid DA-machine learning algorithm in \cite{Carrassi}.
MAOOAM is also just one instance in a wider class of models of intermediate complexity based on QG equations (see \cref{sec:lit}).


Mathematically, our system is a mixture of a system of 2D-Navier-Stokes equations and nonlinear reaction-diffusion equations with extra transport terms. In the latter the nonlinearity comes from the infrared radiation given by a power of the temperature field. Most techniques used are quite standard and are hence only sketched in this paper, for more details see the forthcoming PhD thesis \cite{mythesis}. We will point out in more details only where the treatment differs from the standard approach. Often it is essential that in a-priori estimates terms coming from the nonlinearities either cancel or have the correct sign, putting also restrictions on the boundary condition one can treat. Beside that, for example, the argument for weak-uniqueness is more delicate than it may seem at a first glance. As we will show the ocean temperature field is less regular than the other unknown fields, making the nonlinearity coming from the long-wave radiation hard to control. Though, the outgoing long-wave radiation has a leading term with the expected sign, which can be used to counterbalance other terms. The incoming long-wave radiation describing the heating of the ocean by the atmosphere can only be treated by carefully using regularity of all components of the system. This is a returning theme in the analysis of this model: in order to control the solution one has to use several components and their particular structure in a non-trivial way. Whenever that happens, we give more details. Note that the weak-uniqueness does not seem to hold in the aforementioned (EBM) model due to the absence of the atmospheric streamfunctions, see \cite{LucCan2,EBM_our}. We prove the well-posedness of the Weak and Strong solutions. 

We also show that the solutions are Lipschitz-continuous with respect to the initial condition and to all the forcing term. In particular, we show that the solution depends Lipschitz-continuously on the absorptivity of the atmosphere, which in turn depends on the concentration of greenhouse gases. In this case the argument we use is more delicate since it depends on the form of the nonlinearity of the long-wave radiation and cannot be applied, for example, to the absorptivity of the ocean.

We, furthermore, show the existence of a global attractor, its finite dimensionality and the injectivity of the dynamics on the attractor. These results imply that, despite the infinite-dimensional nature of the governing equations, the long-term behaviour of the system is effectively finite dimensional and and shadowed by its evolution on the attractor.

Another interesting and non-trivial fact is that we show a Foias-Prodi type estimate for our model, that implies the existence of finitely many determining modes, which can be chosen such that they do not depend on the ocean temperature. Estimates of such nature which allow the control of some variables of the model by other have been shown for many fluid dynamics models (see also below). In particular, we mention \cite{Farhat2017} where for a system involving 2D incompressible Navier-Stokes coupled with the temperature equation the convergence for large times of a  data assimilation approach with measurements only on the velocity and not on the temperature. 
Different data assimilation techniques following a variational approach can be used to find initial datum for coupled atmosphere-ocean systems which minimize the distance to a given set of time-distributed observations, as shown for example in \cite{Korn2018},\cite{Korn2021}.
However, the Foias-Prodi estimates used in \cite{Farhat2017} will not be sufficient for our model. In fact, the radiative interaction between atmosphere and ocean creates a state dependent forcing, absent in \cite{Farhat2017}. These and other forcing terms will have to be carefully balanced with the linear terms describing the heat conduction between the layer. The exact numerical values of the parameters of the system will play a crucial role. Using estimates like in  \cite{Chueshov} we can show the existence of determining modes which do not depend, for example, on the streamfunction of the ocean; however, this holds,  as in \cite{Chueshov},for numerical values of the parameters which are essentially large than physically realistic ones. However, we will show that our conditions for determining functions independent of the ocean temperature are fulfilled for the numerical values in the original model.

Existence of solutions with higher regularity than strong solutions is left open, as the nonlinearity does not lie in the domain of the underlying Laplace operator. Another open question is whether initially non-negative temperature profiles stay non-negative under the time development.

\subsection{Previous results on related models}\label{sec:lit}
The multi-layer QG-equations, that constitute the dynamic part of our model, are mathematically very close to 2D-Navier-Stokes equation in vorticity formulation. For multi-layer QG-models, \cite{bernierthesis}, \cite{BernierNS} consider general regular domains with Dirichlet boundary conditions and show well-posedness and existence of the global attractor. In \cite{Chen} existence and uniqueness are proved in the inviscid case with no-flux boundary conditions, and in \cite{BennettKloeden},\cite{OnicaPanetta1}, \cite{OnicaPanetta2} well-posedness is established for the periodic boundary case.
The existence of a finite number of determining functionals for the multi-layer QG model has been shown in \cite{BernierChueshov}. In \cite{Chueshov}, for the two-layer case (with a random forcing on the top layer) the existence of a finite number of determining functionals is shown which only depends on the top layer for large enough friction. 
Moreover \cite{Chueshov} considers a two-layer model with random forcing on the top layer and establishes existence of a finite number of determining functionals, which in particular can be showed to only depend on the top layer for large enough values of the bottom friction.
In more recent works the ergodicity for quasi-geostrophic equations has been treated with random forcing, see \cite{CarigiKunaBroc}, \cite{ButoGroLuoRov}. Quasi-geostrophic equations have been also studied in three dimensions, where “infinitely-many" layers are considered, see \cite[Section 1]{Chen} for an exhaustive list of references about this topic. 

The model MAOOAM under consideration here is positioned in the hierarchy between the EBM and models based on the primitive equation. There are mathematical results for both types of models. If in MAOOAM one considers only the thermodynamic equations which do not depend on the longitude, one arrives at the EBM model with two temperature layers considered in \cite{LucCan}, \cite{LucCan2}.
The well-posedness for an EBM-type model has been discussed first in \cite{HetSch}, where existence of a global solution and the global attractor are shown for a more general system of quasi-linear reaction-diffusion system, see \cite{Diaz} for further reference. The stability of multiple equilibria has been studied for the single-layer model in \cite{Ghil76} taking into account latitude as the only spatial variable, while in \cite{DiazHernTello} is discussed the multiplicity of equilibrium solutions in a two-dimensional EBM according to different values of the solar radiation. Several results have also been proved for two-layer EBM, 
that is, with the same radiation terms as in the model considered here: for the 0D model the stability of the equilibria has been discussed in \cite{LucCan} which will be extended to the 1D model in the paper in progress \cite{LucCan2}.

The primitive equation (PE) has been studied extensively, see for an overview \cite{LiTiti} and for well-posedness results using the theory of evolution equations \cite{HH20}.
Also the existence of a finite dimensional global attractor for the viscous primitive equations has been discussed in \cite{NingTem2014}, \cite{Ning2020}, \cite{ZhouGuo} while in \cite{YuanPei},\cite{ChueshovSqueez} is studied the finite-dimensional dynamics of the solution of PE using respectively a data assimilation algorithm and the squeezing property in order to show the existence of a finite number of determining modes. 

Models based on the primitive equations containing additional physical features have been proposed over the years, yet to the best of our knowledge similar works do not contain both atmosphere and ocean as well as their temperatures with radiation balance. In fact, except in the CAO-model, see \cite{BinzHieber} and references therein, either only the ocean or the atmosphere are modelled by a primitive equation and in the CAO model radiation balance is not included. In \cite{Sarto}, the surface of an ocean model is coupled with an EBM-model, but the atmosphere is not considered. Further, \cite{Korn21} and related works study an ocean model with nonlinear thermodynamics, and in \cite{LiTitiTropical, Doppler} and references within, additional features of the atmosphere like moisture are studied leading to Heaviside nonlinearities.



Finite dimensional dynamics has been also studied for many other physical models: we can cite for example \cite{Farhat2015}, where a finite number of degrees of freedom is showed using a data assimilation algorithm for 2D incompressible Navier-Stokes equations using measurements of only one component of the velocity field. The same approach is used in \cite{3DBernard} for 3D Bénard convection in porous media, a model based on incompressible Euler equation without the transport term coupled with the first law of thermodynamics, to show that the long time dynamics can be determined using only measurements on the temperature, and in \cite{Farhat2017} for a one-layer 2D coupled system which takes into account mechanics and thermodynamics. For an overview about the existence of finite number of determining parameters for the main nonlinear dissipative systems we refer to \cite{AzTiti}.
MAOOAM can be seen as a generalization of the model considered in \cite{Farhat2017} since it is a multi-layer model in which friction and heat exchange between the layers play a fundamental role. Data assimilation algorithms are based on the approach, used for example in \cite{BernierChueshov}, \cite{Chueshov} in the QG framework (respectively deterministic and stochastic) and in the present work in Section~\ref{detmod}, consisting on showing the existence of a finite set of natural parameters of the problem that uniquely determine the asymptotic behaviour of the system. In the data assimilation framework, one takes the physical observations in order to build a linear interpolant operator (for example the projection onto the low Fourier modes), which will be used to estimate the error between the solution of the algorithm and the exact reference solution of the system in terms of the error in the measurements (see \cite{FoiasTiti} for more details).

\section{Physical description of the model, geometry of the domain and boundary conditions}
\subsection{The equations of the model}
\label{physicalsection}
In this section we want to describe the equations of our model, as presented in \cite{MAOOAM}.
The equations for the atmosphere were first proposed by Charney and Straus in \cite{CharneyStraus}: their model is derived applying a scaling argument on the momentum and continuity equations, together with the hydrostatic balance and the first law of thermodynamics, in a two-layer framework (see also \cite[Chapter 6]{Pedlosky} for an introduction on this topic) where the vertical coordinate is discretized (as in \cite{CharneyPhillips}).

For the ocean's streamfunction equation we refer to Pierini \cite{Pierini}, where the ocean is represented using two layers: in the one above the shallow-water approximation is taken into account, while the one below is infinitely deep and quiescent, and then does not play a role in the description of the motion. In the layer above the dissipative effects of the friction are also considered: for a complete derivation of the shallow water equations see again \cite[Chapter 3]{Pedlosky}. Finally, the evolution of the ocean's temperature is described by an energy balance model coupled with the atmosphere (see for example \cite{BB}).

Starting from the vorticity formulation of Navier Stokes equations, and hence using the streamfunctions $\psi_a^{(1)}$ and $\psi_a^{(3)}$ to represent the velocity fields of the two atmospheric layers and $\psi_o$ for the oceanic flow, we can write the system of equations which describes the mechanical interaction between atmosphere and ocean as follows:
\begin{align}
\label{Str}
\pdv{}{t}
\begin{pmatrix}
\Delta \psi_a^{(1)} \\
\Delta \psi_a^{(3)} \\
\Delta \psi_o -\frac{1}{L^2_R}\psi_o
\end{pmatrix}
&+\begin{pmatrix}
J(\psi_a^{(1)},\Delta \psi_a^{(1)}+\beta y)\\
J(\psi_a^{(3)},\Delta \psi_a^{(3)}+\beta y)\\
J(\psi_o,\Delta \psi_o+\beta y)
\end{pmatrix}
=\begin{pmatrix}
\frac{f_0}{p_{\delta}}\\
-\frac{f_0}{p_{\delta}}\\
0
\end{pmatrix}\omega +\nu_S\begin{pmatrix}
\Delta^2 \psi_a^{(1)}\\
\Delta^2 \psi_a^{(3)}\\
\Delta^2 \psi_o
\end{pmatrix} \notag\\
&+\begin{pmatrix}
-k'_d &k'_d &0\\
k'_d &-k'_d-k_d &k_d\\
0 &\frac{C}{\rho h} &-\frac{C}{\rho h}-r
\end{pmatrix}
\begin{pmatrix}
\Delta \psi_a^{(1)} \\
\Delta \psi_a^{(3)} \\
\Delta \psi_o
\end{pmatrix}, \eqname{OAM1}
\end{align}
where $J(u,v)=\partial_x u\ \partial_yv-\partial_y u\ \partial_x v$ is the Jacobian operator.
In absence of forcings and dissipative effects, namely when the right hand side is zero, \ref{Str} represents the conservation of the quasi-geostrophic vorticities $\Delta \psi_a^{(1)}+\beta y$ and $\Delta \psi_a^{(3)}+\beta y$ for the atmosphere and of the potential vorticity $\Delta \psi_o+\beta y -\frac{1}{L^2_R}\psi_o$ for the ocean. Note that the definition the ocean potential vorticity is slightly different from the atmosphere ones because of the term $\frac{1}{L^2_R}\psi_o$ which is related to the shallow water approximation, where $L_R$ is the reduced Rossby deformation radius. The positive constant $\beta$ is related to a simplified way to represent the Coriolis parameter known in geophysics as the $\beta$-plane approximation (see Vallis \cite{Vallis}, Section 2.3.2).

Let us now discuss the terms on the right hand side of \ref{Str}: the first one is related to the vertical velocity function $\omega$ expressed in pressure coordinates, multiplied with $\frac{f_0}{p_{\delta}}$, where $f_0$ is the Coriolis constant and $p_{\delta}$ is the difference between the pressure levels in which $\psi_a^{(1)}$ and $\psi_a^{(3)}$ are defined. The presence of $\omega$ on the right-hand side of \ref{Str} is a consequence of the scaling argument used to derive the QG equations from the Navier-Stokes equations (see e.g. \cite[Section 6.3]{Pedlosky}) combined with vertical discretization. The function $\omega$ is also an unknown of our system, together with $\psi_a^{(1)}$, $\psi_a^{(3)}$ and $\psi_o$, hence we will need another equation to close the system.

With respect to the original model presented by Charney and Straus in \cite{CharneyStraus} and developed by Vannitsem et al. in \cite{MAOOAM},\cite{VannitsemGhil} we also considered diffusive terms given by Laplacians of the vorticities multiplied by the viscosity coefficients $\nu_S$, collected in the second term on the right-hand side of \ref{Str}. These terms, which are frequently used by meteorologists in climate models, are related to the eddy viscosity of the system, which represents momentum transport to larger scales by turbulent eddies. We refer to \cite[Chapter 4]{Pedlosky} for an introduction on turbulent and viscous flows. For our purposes, we consider eddy viscosity terms in order to increase the regularity of solutions, as we will show in the next sections. 

Finally, the third term on the right-hand side is a matrix of friction coefficients between the layers. The coefficients $k'_d$ and $k_d$ quantify the friction respectively between the two atmospheric layers and between the ocean and the atmosphere. The impact of the wind stress on the ocean is given by $\mathcal{W}=d\Delta (\psi_a^{(3)}-\psi_o)$, where $d$ is the drag coefficient of the mechanical ocean-atmosphere coupling and can be expressed as $d=\frac{C}{\rho h}$, with $\rho$ and $h$ respectively density and depth of the fluid layer and $C$ constant related to the curl of the wind stress, for more details see \cite{24}.

Let us consider now the equations for the temperatures of the atmosphere $T_a$ and of the ocean $T_o$. To understand how the two systems interact we have to take into account the long-wave radiation given by the Stefan-Boltzmann law: the atmosphere absorbs the heat coming from the ocean and it emits itself into the ocean and outer space. Therefore the atmospheric total radiation is expressed as:
\[ \mathcal{Q}_a= \epsilon_a\sigma_B T_o|T_o|^3-2\epsilon_a\sigma_B T_a|T_a|^3, \]
where $\sigma_B$ is the Stefan-Boltzmann constant and $\epsilon_a\in (0,1]$ is the emissivity of the atmosphere, taking into account that the atmosphere emits as a grey body rather than a black body. Since we do not know a priori if the temperature is positive, we will take for convenience the quantity $T|T|^3$ instead of $T^4$ in order to consider also the (physically unrealistic) case of negative temperature in Kelvin degrees. 

In the same way, the ocean absorbs the long-wave radiation emitted by the atmosphere, but also almost all incident radiation and it has poor reflective properties. Therefore it emits approximately as a black body (i.e. with emissivity $\epsilon_o\approx 1$), and its radiation balance reads 
\[ \mathcal{Q}_o= \epsilon_a\sigma_B T_a|T_a|^3-\sigma_B T_o|T_o|^3. \]
Other relevant thermodynamic effects are the heat transfer between ocean and atmosphere, and the energy introduced by the short-wave radiation. The heat transfer at the ocean-atmosphere interface is modelled by $\lambda(T_a-T_o)$, where $\lambda$ is a positive coefficient that combines both the latent and sensible heat fluxes. The short-wave radiation terms denoted $R_a(T_a)$ and $R_o(T_o)$ describe the external forcings on the system, as for example the solar radiation fluxes entering in the atmosphere and the ocean. 
In the model studied in \cite{MAOOAM},\cite{VannitsemGhil} the radiation functions are independent of the related temperatures. We choose to consider a more general short-wave radiation which depends also on the temperature. Indeed, starting from the pioneering works of Budyko \cite{Budyko} and Sellers \cite{Sellers} about Energy-Balance Models (EBM), a common choice for the short-wave radiation is a function of the form
\begin{align}
\label{defswr}
    R(t,x,y,T(t,x,y))=R^{(1)}(t,x,y)+R^{(2)}(t,x,y)F(T).
\end{align}
Here $F(T)$ models the effective co-albedo (which represents the fraction of sunlight that is absorbed by the surface of the Earth) and $R^{(1)}(t,x,y)$, $R^{(2)}(t,x,y)$ can take the following form
\[  R^{(i)}(t,x,y)=r_i(t)q_i(x,y), \quad i=1,2\]
where $r_i(t)$ are positive (and possibly periodic) functions allowing to model seasonal cycles and $q_i(x,y)$ are functions that represent insolation.
Classic examples of co-albedo functions typically used in the literature are linear (see for example \cite{Sellers}), piecewise constant or bounded functions (see \cite{North}).  In our discussion the specific form of the radiation will play almost no role, since we will only use some regularity conditions on $R_a, R_o$ that will be specified gradually. 
Therefore, starting from the equations of conservation of potential temperature for the atmosphere and the ocean, we can write the complete equations for $T_a$ and $T_o$ (see again \cite{VannitsemGhil},\cite{MAOOAM}):
\begin{equation}\eqname{OAM2}\begin{split}
\label{Temp}
\pdv{}{t}
\begin{pmatrix}
\gamma_a T_a \\
\gamma_o T_o
\end{pmatrix}
+
&\begin{pmatrix}
\gamma_a J\left(\frac{\psi_a^{(1)}+ \psi_a^{(3)}}{2}, T_a\right)\\
\gamma_o J(\psi_o, T_o)
\end{pmatrix}
=
\begin{pmatrix}
-\lambda -2\epsilon_a\sigma_B |T_a|^3 & \lambda +\epsilon_a\sigma_B |T_o|^3\\
\lambda +\epsilon_a \sigma_B |T_a|^3 & -\lambda - \sigma_B |T_o|^3
\end{pmatrix}
\begin{pmatrix}
T_a\\
T_o
\end{pmatrix} \notag\\
&+
\begin{pmatrix}
R_a(T_a)\\
R_o(T_o)
\end{pmatrix}
+
\begin{pmatrix}
\gamma_a\sigma \frac{p}{R_*}\\
0
\end{pmatrix}\omega+\tilde{\nu}_T
\begin{pmatrix}
\Delta T_a\\
\Delta T_o
\end{pmatrix},\ 
\end{split}\end{equation}
where $\gamma_a$, $\gamma_o$ are the heat capacities of the atmosphere and the active ocean layer, $\sigma$ is the static stability of the atmosphere (taken as a constant), $p$ is a fixed reference value of the pressure and $R_*$ is the perfect gas' constant. Also in this case we need to take into account in the atmospheric equation the vertical velocity $\omega$, which in this case is related to the expression in pressure coordinates of the material derivative of $T_a$. In the equations for $T_o$ and $\psi_o$ we do not have vertical velocity since we consider just one layer.
In equations \ref{Temp} we have again taken into account additional diffusive effects by introducing Laplacians of the temperature fields, which intensity is modulated by the coefficient $\tilde{\nu}_T$ and are technically advantageous for what we want to do.

Finally, the last equation of the model consists in a relation between atmospheric temperature and streamfunction which close the system of equations \ref{Str}, \ref{Temp}.
Taking $T_a^{(0)}$ a constant reference value in space and time for the temperature, we can consider the fluctuations of the atmospheric temperature $\delta T_a$ around $T_a^{(0)}$ discretizing the derivative of the streamfunction with respect to the pressure as follows
\begin{align}
    \label{defTfin}
    T_a=T_a^{(0)}+\delta T_a=T_a^{(0)}+\frac{pf_0}{R_* p_{\delta}}(\psi_a^{(1)}-\psi_a^{(3)}). \eqname{T}
\end{align}
One possible choice to simplify the above expression is to consider, as for example in \cite{MAOOAM}, the atmospheric temperature at the isobaric interface between the two layers, so that $p=p_{\delta}=$500 hPa.

We can take a similar decomposition for $T_o$: \[T_o=T_o^{(0)}+\delta T_o\] but for the ocean we do not have two distinct streamfunctions (because the bottom layer is motionless) and then we cannot represent $\delta T_o$ as a difference between them.

We will call the system composed by equations \ref{Str}, \ref{Temp} and \ref{defTfin} by the acronym MAOOAM, standing for \emph{Modular Arbitrary-Order Ocean-Atmosphere Model}, following the nomenclature given in \cite{MAOOAM}. The words \emph{Modular Arbitrary-Order} follows from the fact that an arbitrary number of Fourier modes is used in \cite{MAOOAM} to discretize the model. The unknowns of MAOOAM are $\psi_a^{(1)},\,\psi_a^{(3)},\,\psi_o,\, T_o$ and $\omega$, but we can easily get rid of $\omega$ using relation \ref{defTfin} and putting the equations for $\psi_a^{(1)},\,\psi_a^{(3)}$ and $T_{a}$ together. Therefore we can consider \ref{Str}, \ref{Temp} a system of four equations with four unknowns.

The model can also be expressed using two different variables for the atmospheric streamfunctions, namely the so called barotropic and baroclinic streamfunctions that we denote as $\psi_a^t$ and $\psi_a^c$ respectively and define as 
\begin{align}
    \label{Jcambcord1}
    \psi_a^t:=\frac{\psi_a^{(3)}+\psi_a^{(1)}}{2}, \qquad \psi_a^c:=\frac{\psi_a^{(3)}-\psi_a^{(1)}}{2}.
\end{align}
 One reason for this choice is that in absence of thermodynamic effects the atmosphere would be described by the barotropic streamfunction alone (see e.g. \cite{Pedlosky}). In the model considered, this is exemplified by the fact that the atmospheric temperature is proportional to the baroclinic streamfunction by \ref{defTfin} i.e.
 \begin{equation}\label{eq:Tb}
     T_a=T_a^{(0)}+\delta T_a=T_a^{(0)}-\frac{2pf_0}{R_* p_{\delta}}\psi_a^c. \eqname{T-b}
 \end{equation}
 We can rewrite \ref{Str} in terms of the barotropic and baroclinic streamfunctions as follows:
\begin{align}
    \pdv{\Delta\psi_a^t}{t}
    +J(\psi_a^t,\Delta\psi_a^t)
    +J(\psi_a^c,\Delta\psi_a^c)
    +\beta\pdv{\psi_a^t}{x}
    &=-\frac{k_d}{2}\Delta(\psi_a^t+\psi_a^c-\psi_o)
    +\nu_S\Delta^2\psi_a^t, \notag\\
    \begin{split}
    \pdv{\Delta\psi_a^c}{t}
    +J(\psi_a^c,\Delta\psi_a^t)
    +J(\psi_a^t,\Delta\psi_a^c)
    +\beta\pdv{\psi_a^c}{x}
    &=-\frac{k_d}{2}\Delta(\psi_a^t+\psi_a^c-\psi_o) \\
    &\quad -2k'_d\Delta\psi_a^c
    -\frac{f_0}{p_{\delta}}\omega
    +\nu_S\Delta^2\psi_a^c. 
    \end{split} \label{bb} \eqname{OAM1-b}\\
     \pdv{}{t}\left(\Delta\psi_o - \frac{1}{L_R^2}\psi_o\right)
     + J(\psi_o,\Delta\psi_o)
     +\beta\pdv{\psi_o}{x} \notag
     &=   \frac{C}{\rho h} \Delta(\psi_a^t+\psi_a^c-\psi_o)
     - r \Delta\psi_o
     + \nu_S \Delta^2 \psi_o .\notag
\end{align}
The only modification in \ref{Temp} is recognising that the atmosphere temperature is advected by $\psi_a^t$, namely we have $J(\psi_a^t, T_a)$. Then the MAOOAM system can be equivalently expressed by the equations \ref{bb},\ref{Temp} and \ref{eq:Tb} with unknowns $\psi_a^t, \psi_a^c$, $\psi_o$ and $T_o$. We will practically use this formulation in \cref{sectenergy} and subsequent sections.

\subsection{Geometry of the Domain and Boundary Conditions}
\label{GeoandBC}
In this subsection we will describe the geometry of the spatial domain of the model just introduced and what boundary conditions we need to take. Following \cite{MAOOAM}, the model is defined on a rectangular domain $Q=[0,L]\times[0,\alpha L]$, where $L$ is a horizontal length scale of order $10^6\ m$ and $\alpha\in (0,1)$ is a constant which allows us to take into account the different extensions in the $x$ and $y$ directions. Since in our model we neglect the spatial extension of the continents, atmosphere and ocean must interact on overlapping areas, and then the geometry  of the domains of the atmosphere $\Omega_a$ and of the ocean $\Omega_o$ must be the same. The atmospheric flows, which in this subsection will be denoted simply by $\psi_a$ independently of the layer they belong to, are zonally periodic (i.e. periodic in the $x$-direction) with no-flux boundary conditions in the meridional direction, i.e. \[ \restr{\pdv{\psi_a}{x}}{y=0,\alpha L}=0,\] whereas the oceanic flow $\psi_o$ has no-flux in both meridional and zonal directions, i.e.
\[ \restr{\pdv{\psi_o}{x}}{y=0,\alpha L}=0,\quad \restr{\pdv{\psi_o}{y}}{x=0,L}=0.\]
A schematic representation of these boundary conditions on $\Omega_a$ and $\Omega_o$ is given in \cref{fig:domain}.

\begin{figure}
\begin{tikzpicture}[scale=1, line join=round, line cap=round]

\def\L{6}
\def\H{1.2}
\def\s{1.2}

\tikzset{
  box/.style={thick},
  flux/.style={gray, thick},
  lab/.style={font=\small}
}

\begin{scope}[shift={(0,0)}]
\draw[box]
    (0,0) -- (\L,0) -- (\L+\s,\H) -- (\s,\H) -- cycle;

  \node at (3.2,0.6) {$\Omega_o$};
  \node[lab, below] at (3,0) {no flux};
  \node[lab, above] at (3.8,1.2) {no flux};
  \node[lab, left] at (0.1,0.6) {no flux};
  \node[lab, right] at (7,0.6) {no flux};
   \node[lab, left] at (-1.2,0.6) {$\psi_o$};

\begin{scope}[shift={(9,-0.1)}, scale=0.9]
  \draw[->, thick] (0,0) -- (1.2,0) node[below] {$x$};
  \draw[->, thick] (0,0) -- (0,1.2) node[left] {$z$};
  \draw[->, thick] (0,0) -- (-0.8,-0.6) node[below left] {$y$};
\end{scope}

\end{scope}

\begin{scope}[shift={(0,2.2)}]
\draw[box]
    (0,0) -- (\L,0) -- (\L+\s,\H) -- (\s,\H) -- cycle;

  \draw[flux] (0,0) .. controls (3,0.6) .. (\L,0);
  \draw[flux] (\s,\H) .. controls (3.8,1.8) .. (\L+\s,\H);

  \node at (3.2,0.6) {$\Omega_a$};
  \node[lab, left] at (0.4,0.6) {periodic};
  \node[lab, right] at (7,0.6) {periodic};
  \node[lab, above] at (3.8,1.2) {no flux};
  \node[lab, below] at (3,0) {no flux};

  \node[lab, left] at (-1.2,0.6) {$\psi_a^{(1)}$, $\psi_a^{(3)}$};
\end{scope}


\end{tikzpicture}\caption{Schematic representation of the domains and boundary conditions for the variables $\bm{\psi}$.}\label{fig:domain}
\end{figure}
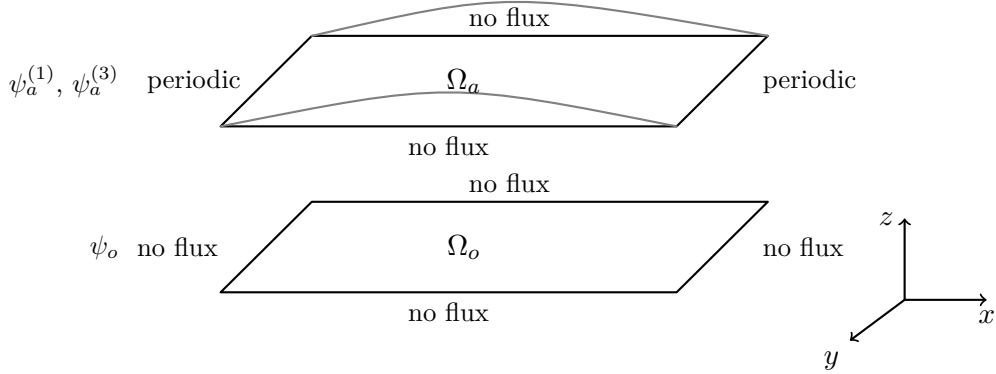

These conditions are used to represent the zonal wind in the atmosphere and the absence of global horizontal circulation of water in the ocean. The latter is a consequence of the presence of continents that block the circulation: in our model we do not consider land, hence continents can be assumed infinitely thin. These boundary conditions can be also expressed in the following way
\begin{equation}
\label{no-flux}
\nab \psi \times \hat{n}=0 \ \text{ on }\Gamma,
\end{equation}
where $\Gamma=\partial Q$. In particular, $\Gamma_a$ is made by two disjoint circumferences and $\Gamma_o$ by a rectangle. Condition \eqref{no-flux} reflects the vanishing geostrophic motion on the boundary, and it is used in Pedlosky \cite{Pedlosky}, Section 6.9, as boundary condition for the ocean. In our case it holds also on the two horizontal circles bounding the atmosphere.
Now, since \eqref{no-flux} implies that $\psi$ on $\Gamma$ depends only on time, assuming that the restriction to the boundary of $\psi_a,\psi_o$ is independent of $t$, we immediately get that $\psi_a,\psi_o$ are constant on $\Gamma$. In particular we will choose 
\begin{align}
\label{dirboundpsi}
\psi_a=0,\quad \psi_o=0\ \text{ on }\Gamma
\end{align} 
in order to neglect the boundary contributions when we will derive the energy inequality \ref{enes1} in Section~\ref{sectenergy}.
It is worth noting that this choice is not compatible with the eigenfunctions considered in \cite{MAOOAM}.

Summing up, we will take the rectangle $Q$ defined above and we will consider as atmospheric domain $\Omega_a=\bigslant{Q}{\{0\sim_x L\}}$, i.e. with the identification of the sides $x=0$ and $x=L$ in order to take into account the periodicity in the zonal direction, while for the ocean we will simply take $\Omega_o=Q$.
On the boundaries $\Gamma_a$ and $\Gamma_o$ we will have the conditions $\psi_a=\psi_o=0$. Since the highest order term in the derivatives of \ref{Str} is given by the biharmonic operator $\Delta^2$ we need also boundary conditions for the derivatives of $\psi_a$ and $\psi_o$. We choose $\Delta\psi_a=\Delta\psi_o=0$ on $\Gamma_a$ and $\Gamma_o$ in order to make $\Delta^2$ self-adjoint. Regarding the temperatures, we can use the definition of atmospheric temperature \ref{eq:Tb} and $\restr{\psi^c}{\Gamma_a}=0$ to get $\restr{T_a}{\Gamma_a}=T_a^{(0)}$. For $T_o$ in general there is no relation with $\psi_o$, but we can choose $\restr{T_o}{\Gamma_o}=T_o^{(0)}$ in analogy with $T_a$. 
In the rest of the article we will just write $\Omega$ to denote both $\Omega_a$ and $\Omega_o$. 

\section{Energy Equations}
\label{sectenergy}
In this section we will derive a fundamental energy estimate that will be used frequently throughout this work. 
The inequality involves several terms, so we will split its derivation in two parts: first we will present an estimate which takes into account only the equations for the streamfunctions. The estimate obtained will depend on the unknown $\omega$: then, taking into account the equations for the atmospheric and ocean temperature $T_a$ and $T_o$, we will be able to eliminate $\omega$. We highlight that the derivation is only formal as we have not established well-posedness of the equation yet, but this estimate will be made rigorous when considering the Galerkin approximation for the equation.

Let us first set some notation. We will use $|u|_{L^p}$ for the standard $L^p(\Omega)$ norm and, for simplicity of notation, we will denote the $L^2(\Omega)$ norm and scalar product respectively $|\cdot|$ and $(\cdot,\cdot)$. Next $H^1_0$ denotes the Sobolev space of $L^2$ functions which vanish on the boundary with weak derivative also in $L^2$ and $H^k$ denotes the space of $k$-times weakly differentiable functions with derivatives in $L^2$, see for example \cite{Brezis}, with norms respectively 
\[ 
    \| u\|_{H^1_{0}(\Omega)}=|\nab u|,\ \quad \Vert u\Vert_{H^k}=\sum_{|\alpha|\le k} |D^{\alpha} u|.
\]
We take the self-adjoint operator $A:=-\Delta$ with Dirichlet boundary conditions.
By the equivalence of the $H^2$ and $D(A)$ norms one has $c\|u\|_{H^2(\Omega)}\le |\Delta u|=:|u|_{D(A)}$ when $u\in D(A)=H^2\cap H_0^1$.
We will denote by $D(A)^{*}$ the dual of $D(A)$, by $H^{-k}$ the dual of $H_0^k$ and by $\langle\cdot,\cdot\rangle$ the dual relation, while for the scalar product in $H_0^1$ we use $(\!(\cdot,\cdot)\!)$.

Consider the first two equations in \ref{bb} for the barotropic and baroclinic streamfunction and take the $L^2$-product with the corresponding streamfunctions. For the third equation in \ref{bb} in the ocean streamfunction $\psi_o$ we take $L^2$-product with $\kappa \psi_o$ where $\kappa=\frac{\rho h k_d}{C}$. This rescaling is done in order to get the same constants on the damping terms in all the equations in order to combine them. We then add the three resulting equations.

Under sufficient regularity, we can express some of the quantities in the equations \ref{bb} containing a divergence-type part as follows:
\begin{align}
&\psi\pdv{\Delta\psi}{t}=\nabla\cdot \bl(\psi\nabla\pdv{\psi}{t}\br)-\nabla\psi\cdot\nabla \pdv{\psi}{t}=\nabla\cdot \bl(\psi\nabla\pdv{\psi}{t}\br)-\frac{1}{2}\pdv{\vert\nabla\psi\vert^2}{t}, \nonumber
\\ & \label{Jpropfond}
J(\psi,\xi)\varphi=\nab\cdot(\xi\varphi\nab^{\perp}\psi)-J(\psi,\varphi)\xi=\nab\cdot(\xi\varphi\nab^{\perp}\psi)+ J(\varphi,\psi)\xi,
\end{align}
where $\psi$ is a generic streamfunction for the atmosphere or the ocean and $\varphi$, $\xi$ are regular functions.
Then the Jacobian terms vanish given the boundary conditions $\restr{\bm\psi}{\partial\Omega}=\restr{\Delta\bm\psi}{\partial\Omega}=0$, where $\bm\psi:=(\psi_a^t,\psi_a^c,\psi_o)$. This leads to the following expression:
\begin{multline}
\label{123}
    \frac{1}{2}\pdv{}{t}\biggl\{
    \vert\nabla \psi_a^t\vert^2
    +\vert\nabla \psi_a^c\vert ^2
    +\kappa\biggl(
        \vert\nabla \psi_o\vert ^2
        +\frac{|\psi_o|^2}{L^2_R}
    \biggr)\biggr\}\\
    +2k'_d \vert\nabla \psi_a^c \vert ^2
    +k_d|\nabla(\psi_a^t+\psi_a^c-\psi_o)|^2
    +r\kappa|\nabla\psi_o|^2\\
    +\nu_S\bigl(
        \vert \Delta \psi_a^t\vert^2
        +\vert \Delta \psi_a^c\vert^2
        +\kappa\vert \Delta \psi_o\vert^2
    \bigr)
    =\frac{f_0}{p_{\delta}}\ \int_{\Omega}\ \psi_a^c\ \omega\ d\tilde{A},
\end{multline}
where the first line is the variation in time of kinetic energy and gravitational potential energy of the ocean, the second takes into account the loss of energy due to the friction between the layers and with the bottom of the ocean, the third is related to the eddy viscosities, and the last term involving $\omega$, where $d\tilde{A}=dx\wedge dy$ is the usual Lebesgue measure in $\bbR^2$, is the only driving term describing the heat exchange in the atmosphere.

Next we crucially observe that the quantity $\int_{\Omega}\psi_a^c\ \omega\ d\tilde{A}$ in \eqref{123} can be expressed in terms of $T_a$ and $T_o$ taking the $L^2$ scalar product of the first equation in \ref{Temp} with $T_a-T_a^{(0)}$ and of the second one with $T_o-T_o^{(0)}$. Here we subtract the mean temperatures to obtain a function which is zero on the boundary, making the boundary terms in the equation vanish.

Moreover taking $\epsilon_a\in (0,1]$ (which is physically realistic), the infrared radiation terms can be rewritten as a negative quantity:
\begin{equation}
\label{relTaTo}
\begin{split}
    &-\sigma_B \vert T_o\vert_{L^5}^5-2\epsilon_a\sigma_B\vert T_a\vert_{L^5}^5+\epsilon_a\sigma_B\bigl((|T_a|_{\bbR}^3T_a,T_o)+(|T_o|^3_{\bbR}T_o,T_a)\bigr)=\\
    &-(1-\epsilon_a)\sigma_B \vert T_o\vert_{L^5}^5-\epsilon_a\sigma_B\vert T_a\vert_{L^5}^5-\epsilon_a\sigma_B\underbrace{\bigl((|T_a|^3_{\bbR}T_a,T_a-T_o)+(|T_o|^3_{\bbR}T_o,T_o-T_a)\bigr)}_{(|T_a|^3_{\bbR}T_a-|T_o|^3_{\bbR}T_o)(T_a-T_o)\ge 0}.
\end{split}
\end{equation}
To control the energy of the incoming radiation terms $R_a(T)$ and $R_o(T)$ 
we assume for both a condition of the form
\begin{align}
    \label{xlpqsxa}
    (R_i(T_i), T_i-T_i^{(0)})\le \mathcal{F}_R(t)&+\frac{1}{10}\epsilon_i\sigma_B|T_i|^5_{L^5}, \quad i =a,\ o
\end{align}
where $\mathcal{F}_R$ is a function integrable locally in time, and $\epsilon_o := 1- \epsilon_a$. 

\begin{rem}
    Requiring \eqref{xlpqsxa} to hold is reasonable for typical formulations of the short-wave radiation. For example assuming the short-wave radiation of the form written in \eqref{defswr} with the co-albedo function linearly-growing in $T$ (as for example in \cite{Sellers}), we can compute for $i = a, o$
    \begin{align*}
(R_i(T_i),T_i-T_i^{(0)})\le C\left(|R_i^{(1)}(t)|_{L^1}+|R_i^{(1)}(t)|_{L^{5/4}}^{5/4}+|R_i^{(2)}(t)|_{L^{5/4}}^{5/4}+|R_i^{(2)}(t)|_{L^{5/3}}^{5/3}\right)+\frac{1}{10}\epsilon_i\sigma_B|T_i|^5_{L^5},
\end{align*}
where we used H\"older and Young's inequalities with $p=\frac{5}{3}$, $q=\frac{5}{2}$ for the term $(R_i^{(1)},T_i)$ and with $p=\frac{5}{4}$, $q=5$ for $(R_i^{(2)} |T_i|,T_i)$. 
\end{rem}
Next we use the damping terms $\epsilon_a\sigma_B|T_a|^5_{L^5}$, $(1-\epsilon_a)\sigma_B|T_o|^5_{L^5}$ obtained in \eqref{relTaTo} to counterbalance the right hand side of \eqref{xlpqsxa}, so that, using standard estimates as H\"{o}lder and Young's inequalities, we can control the remaining quantities in the expression for $\int_{\Omega}\psi^c\omega\ d\tilde{A}$. We can then find an energy bound independent of the temperature functions, but dependent on the constant reference values $T_a^{(0)}$ and $T_o^{(0)}$ (for more details see \cite{mythesis}). 

Now substituting $\omega$ in \eqref{123} and integrating with respect to $t\in [0, t_*]$ we find
\begin{multline}
    \frac{1}{2}\biggl\{\vert\nabla \psi_a^{t}(t)\vert ^2+\vert\nabla \psi_a^{c}(t)\vert ^2+\kappa\bl(\vert\nabla \psi_o(t)\vert ^2+\frac{1}{L^2_R}|\psi_o(t)|^2\br)\\
    +\mu(\gamma_a\vert T_a(t)-T_a^{(0)}\vert^2+\gamma_o\vert T_o(t)-T_o^{(0)}\vert^2)\br\}\\
    +\nu_S\int_0^t\bigl(\vert \Delta \psi_a^t(s)\vert^2 + \vert \Delta \psi_a^c(s)\vert^2+\kappa\vert \Delta \psi_o(s)\vert^2\bigr)ds+\nu_T\int_0^t (|\nabla T_a(s)|^2+|\nabla T_o(s)|^2) ds\\
    +\lambda\int_0^t|T_a(s)-T_o(s)|^2 ds + \mu\sigma_B\tilde{\epsilon}_a\bl(\int_0^t|T_a(s)|^5_{L^5}ds+\int_0^t|T_o(s)|^5_{L^5}ds\br)\\
    +2k'_d \int_0^t\vert\nabla \psi^c_a(s)\vert ^2ds+k_d\int_0^t\vert\nabla (\psi^t+\psi^c-\psi_o)(s)\vert ^2ds + r\kappa\int_0^t\vert\nabla\psi_o(s)\vert ^2 ds\\
     \le \frac{1}{2}\biggl\{\vert\nabla \psi_a^t(0)\vert ^2+\vert\nabla \psi_a^c(0)\vert ^2+\kappa\bl(\vert\nabla \psi_o(0)\vert ^2+\frac{1}{L^2_R}|\psi_o(0)|^2 \br)\\+\mu(\gamma_a\vert T_a(0)-T_a^{(0)}\vert^2 +\gamma_o\vert T_o(0)-T_o^{(0)}\vert^2)\br\}+E(t),
\label{enes1}
\eqname{En.In.}
\end{multline}
where we defined the constant parameters 
\begin{align*}
    \mu:=\frac{R_*^2}{2p^2 \gamma_a\sigma},
    \qquad 
    \nu_T:=\mu\tilde{\nu}_T,\qquad \tilde{\epsilon}_a:=\frac{7}{10}\min\{\epsilon_a,1-\epsilon_a\}.
\end{align*} 

The inequality \ref{enes1} just derived describes the energy dissipation caused by different terms: in the third line by viscosity, in the fourth by heat exchange and outgoing infrared radiation and in the fifth by the friction between the layers and with the bottom of the ocean. To have an idea of the order of magnitude of the coefficients in the time derivative terms of \ref{enes1} see Table \ref{tabella1}.
\begin{table}
\centering
\begin{tabular}{||c c c||} 
 \hline
 Parameter (Unit) & Value & Physical meaning\\ [0.5ex] 
 \hline\hline
 $\kappa$ & 15 & 
 Intensity of the evolution of $|\nab\psi_o|^2$\\ 
 \hline
 $\frac{\kappa}{L^2_R}\ (m^{-2})$ & $O(10^{-8})$ & Intensity of the evolution of $|\psi_o|^2$ \\
 \hline
 $\mu\ (m^2\cdot kg\cdot K)$ & $O(10^{-6})$ & Thermodynamical scaling factor\\
 \hline
 $\gamma_a\mu\ (m^2\cdot s^{-2}\cdot K^{-2})$ & 9 & Intensity of the evolution of $|T_a|^2$\\
 \hline
 $\gamma_o\mu\ (m^2\cdot s^{-2}\cdot K^{-2})$ & 500 & Intensity of the evolution of $|T_o|^2$\\ [1ex] 
 \hline
\end{tabular}
\caption{Magnitude and physical interpretation of the constants involved in the time-derivative terms of \ref{enes1}.}
\label{tabella1}
\end{table}
Moreover the quantity $E(t)$ contains, beside terms coming from the short-wave radiation like $\mathcal{F}_R$ and the size of the domain $\Omega$, also terms depending on the reference temperature values $T_a^{(0)}, T_o^{(0)}$. 

We observe that our estimate does not show a decreasing behaviour for the energy if we assume the forcing terms $R_a,R_o$ to be null, as one may expect. This is due to the terms containing $T_a^{(0)}, T_o^{(0)}$. We can avoid these terms if we additionally assume the property
\begin{align}
\label{intomega}
\int_{\Omega} \omega\ d\tilde{A}=0,
\end{align}
which can be justified by the same scaling argument used to derive \ref{Str} and \ref{Temp}.
\section{Existence and Uniqueness of the Weak Solution}
\label{Chapweak}
In this section we want to show that the system of PDEs \ref{bb}, \ref{Temp} admits a unique solution in suitable spaces of functions. In order to define classes of solutions to our model, we need families of Banach-space valued functions: these spaces will be denoted by
\begin{align*}
L^p_tH^s_x:=L^p(0, t_*;H^s(\Omega)),
\end{align*} 
where $t_*$ is finite but arbitrary and we 
specify 
when we actually can take $t_*=+\infty$.

For our purposes it will be convenient to denote the cartesian product between three or four Banach spaces as follows: $ \bm{X}=X^3$, $\mathbb{X}=X^4$. Furthermore, we consider the vector of streamfunctions $\bm{\psi}=(\psi_a^{t},\psi_a^{c},\psi_o)$ and the vector $(\bm{\psi},T_o)$.
On $\bm{X}$ we will often use the following weighted norm:
\[  \Vert\bm{\psi}\Vert_{\bm{X}}^2=|\psi_a^t|_X^2+|\psi_a^c|_X^2+\kappa|\psi_o|_X^2, \]
where $\kappa=\frac{\rho h k_d}{C}$ has been introduced in the previous section.
In particular we will denote $\bmD(A) = D(A)^3$ and $\bbD(A) = D(A)^4$.
 
We start by giving a definition of weak solutions, in analogy with the weak solution for Navier-Stokes equations presented for example in \cite{RobinsonNS}:
\begin{defm}[Weak Solution for MAOOAM]
\label{defweak}
We say that a function $(\bm{\psi},T_o)$ is a Weak Solution satisfying the initial condition $(\bm{\psi}(0),T_o(0))\in (\bmH^1_0\times L^2)(\Omega)$ for the system \ref{bb}, \ref{Temp}, \ref{eq:Tb} if
\begin{enumerate}
\item[1)] $\bm{\psi}\in \mathbf{L}^{\infty}_t\mathbf{H}^1_{0,x}\cap \mathbf{L}^2_t\mathbf{H}^2_x$ and $T_o-T_o^{(0)}\in L^{\infty}_tL^2_x\cap L^2_tH^1_{0,x}\cap L^5_tL^5_x$.
\item[2)] $(\bm{\psi},T_o)$ satisfies the weak form of the equations \ref{bb}, \ref{Temp} in the space $C^{1}([0, t_*]; D(A))^{*}$.
\end{enumerate}
Moreover, if also the energy inequality \ref{enes1} holds we will call our solution a \emph{Leray-Hopf} weak solution.
\end{defm}
Note that in the second point of Definition~\ref{defweak} we have to interpret the equation in a weak form, that is the biharmonic operator is interpreted as $(\Delta \psi, \Delta \varphi)$ for all $\varphi\in D(A)$ which express the boundary condition $\restr{\Delta\psi}{\partial\Omega}=0$, and as usual the Jacobian operator has also to be interpreted in a generalized form.
We observe that all the streamfunctions satisfy automatically the no-flux boundary condition $\nab \psi \times \hat{n}=0$  on $\partial\Omega$ as discussed in Section~\ref{GeoandBC}.

To show well-posedness in the weak sense just defined we will require the short-wave radiation forcings to fulfil the following conditions:
\begin{assumption}\label{assum:R:weaksol}
    There exists $\mathcal{L}_{R_a}(t)$ and $\mathcal{L}_{R_o}(t)$ locally integrable in time such that 
\begin{align}
\label{xjgpw}
    &(R_a(\hg)-R_a(\tg), \hg- \tg)\le \mathcal{L}_{R_a}(t)|\nab (\hg - \tg)|^2, \notag\\
    &(R_o(\hg)-R_o(\tg), \hg- \tg)\le \mathcal{L}_{R_o}(t)|\hg - \tg|^2 +\frac{\nu_T}{8}|\nab (\hg - \tg)|^2,
\end{align}
for any $\tg,\, \hg \in H^1$.
\end{assumption}
These requirements are effectively less strict than asking for $R_a$, $R_o$ to be for example globally Lipschitz from $H^1$ into $L^2$ and they would accommodate also for appropriate integral functions. 

We are ready to state the main theorem on this section:
\begin{thm}[Existence and Uniqueness of the Weak Solution for MAOOAM]
\label{teoexuniqweak}
Assume that the short-wave radiation functions $R_a$ and $R_o$ are bounded and continuous in $L^{2}_tL^{2}_x$ and $L^{5/4}_tL^{5/4}_x$, respectively, and they satisfy \eqref{xlpqsxa}. Then for any initial condition $(\bm{\psi}(0),T_o(0))\in (\bmH^1_0\times L^2)(\Omega)$, there exists a Weak Solution 
\[
    (\bm{\psi},T_o)\in C([0, t_*],(\bmH_0^1\times L^2)(\Omega))
\] 
for the system \ref{bb}, \ref{Temp}, \ref{eq:Tb}. Moreover if \cref{assum:R:weaksol} holds then the solution is unique and depends continuously on the initial conditions.
\end{thm}

Assuming that the short-wave radiations $R_a$ and $R_o$ are functions of space and temperature but not of time (i.e. the solar radiation depends on time only through temperature, $R=R(x,y,T(t,x,y)$), an immediate consequence of Theorem~\ref{teoexuniqweak} is the existence of a semidynamical system on $(\bmH_0^1\times L^2)(\Omega)$. In Sections~\ref{QSuniq} and \ref{Strsec} we will show that the solution of MAOOAM can be also considered a semidynamical system on spaces of higher regularity.

We divide the proof of \cref{teoexuniqweak} first showing in \cref{exweak} the existence of solutions and then in \cref{uniq} their uniqueness and continuous dependence on the initial conditions. We close then this section by showing continuous dependence of the solutions with respect to significant parameters of the equations.

\subsection{Existence of Weak Solutions}\label{exweak}
 The approach we use to ensure existence of the Weak Solution is the Galerkin approximation. We will not present the entire argument as it follows standard techniques but we will discuss the more delicate points, namely ensuring that the nonlinear terms are also uniformly bounded in appropriate spaces and the consequent convergence to the original equation. 
 
Let $P_n$ be the projection on the space spanned by the first $n$ eigenfunctions of the operator $A$ and denote 
 \begin{equation}\label{eq:project:var}
     \bm{\psi}_{n} = P_n \bm{\psi}, \quad  T_{i,n} = P_n  T_{i},\quad i = a,o.
 \end{equation}
 Then we apply the projection on the equivalent equations \ref{bb}, \ref{Temp}, \ref{eq:Tb} (recall the unknown are the three streamfunctions and the ocean temperature) and we can show that the energy estimate \ref{enes1} holds also for the variables \eqref{eq:project:var}. In fact the projection commutes with all linear operators and for projected nonlinear terms when taking the $L^2$ product one has for example 
\begin{equation*}
    \left( \psi_{n}^c, P_n J(\psi_n^c,\Delta\psi_n^t)\right)  =  \left( \psi_{n}^c, J(\psi_n^c,\Delta\psi_n^t)\right)
\end{equation*}
giving the same structure we formally treated in \cref{sectenergy}.
 
From the energy inequality \ref{enes1} holding for \eqref{eq:project:var}, it follows that $\bm{\psi}_{n}$ is uniformly bounded in $\bm{L}^{\infty}_t\bm{H}^1_{0,x}$, $T_{a,n}$, $T_{o,n}$ are uniformly bounded in $L^{\infty}_tL^2_x\cap L^5_tL^5_x$, and $\nab T_{a,n}$, $\nab T_{o,n}$ are uniformly bounded in $L^2_t L^2_x$.
Next, one has to ensure that this property extends in appropriate spaces to all the terms in the projected equations. We will focus on the nonlinear terms.

We start by discussing the regularity for the temperature equations \ref{Temp}. 
Using H\"older and Ladyzhenskaya's inequalities one can see that $|T_{a,n}|^3T_{a,n}$ is uniformly bounded in $L^{p}_t L^{q}_x$ for all $p,q \in [1,\infty)$ and using the properties of the operator $P_n$ we can conclude that $P_n|T_{a,n}|^3T_{a,n}$ is uniformly bounded in $L^2_t L^2_x$. The nonlinear transport term $J(\psi_{a,n}^t,T_{a,n})$ can be uniformly bounded in $L^{2}_t L^{2}_x$, and consequently $P_nJ(\psi_{a,n}^t,T_{a,n})$ is bounded uniformly in $L^2_t H^{-1}_x$.
For the short-wave radiation we assumed that $R_{a,n}(T_{a,n})$ is a bounded function with values in $L^2_tL^2_x$.  

Regarding the ocean temperature, as $T_{o,n}$ is uniformly bounded only in $L^{\infty}_t L^2_x\cap L^2_t H^1_x$ we can only get $|T_{o,n}|^3T_{o,n}$ uniformly bounded in $L^{5/4}_tL^{5/4}_x$,
and $P_n|T_{o,n}|^3T_{o,n}$ will be bounded in $L^{5/4}_tH^{-1}_x$. 
To have the same regularity for the short-wave radiation we assumed that $R_{o,n}(T_{o,n})$ is a bounded function in $L^{5/4}_tL^{5/4}_x$.

The nonlinear transport term $J(\psi_{o,n},T_{o,n})$ can be uniformly bounded in $L^{5/4}_t L^{5/4}_x$ (consequently $P_nJ(\psi_{o,n},T_{o,n})$ is bounded uniformly in $L^{5/4}_t H^{-1}_x$) as a consequence of $ \nab\psi_{o,n} \in L^{\infty}_t L^2_x\cap L^2_t H^1_x\subset L^{10/3}_t L^{10/3}_x$ where the inclusion is obtained by $L^p_tL^q_x$-interpolation. Using these properties we deduce that $\partial_t T_{o,n}$ is uniformly bounded in $L^{5/4}_t H^{-1}_x$ and converges also strongly in $L_t^2 L_x^2$ by Aubin-Lions lemma.

Now let us discuss the regularity for the equations \ref{bb} for the barotropic and baroclinic streamfunctions. In the equation for $\psi_n^c$ the less regular terms are $\Delta^2\psi_{n}^c$ and $P_n|T_{o,n}|^3T_{o,n}$. The first is uniformly bounded in $L^2_tD(A)^*_x$ and the second is uniformly bounded in $L^{5/4}_tH^{-1}_x$.
Next, in order to estimate the Jacobians we use the following inequalities:
\begin{align}
    &|(J(u,v),\Delta u)| \le C|\Delta v||\nab u||\Delta u|, &&\quad u,v\in H^2\cap H_0^1,\label{estJ1}\\
    &|(J(u,v),w)| \le C|\nab u||\nab v|^{1/2}|\Delta v|^{1/2}|w|^{1/2}|\nab w|^{1/2}, &&\quad u\in H^1,\ v,w\in H^2\cap H_0^1 \label{estJ2}.
\end{align}
%
These have been shown for example in \cite[Lemma 3.1]{Chueshov} in the periodic case, and they also hold with the boundary conditions used here. In particular to show \eqref{estJ1} we need the following relation, which is a consequence of an integration by parts together with the divergence theorem:
\begin{align*}
    (J(u,v),\Delta u)&=-\int_{\Omega}\frac{\partial^2 v}{\partial x\partial y}\bl\{\bl(\pdv{u}{x}\bl)^2-\bl(\pdv{u}{y}\bl)^2\br\}d\tilde{A}+\int_{\Omega} \pdv{u}{x}\pdv{u}{y}\bl\{\frac{\partial^2 v}{\partial y^2}-\frac{\partial^2 v}{\partial x^2}\br\}d\tilde{A}\\
    &\quad +\frac{1}{2}\int_{\partial\Omega} (\nab^{\perp} v\cdot \nab u)\nab u\cdot \hat{n}\ dl- \frac{1}{2}\int_{\partial\Omega} (\nab v\cdot \nab u)\nab^{\perp} u\cdot \hat{n}\ dl,
\end{align*} 
where the boundary terms vanish due to the no-flux condition.

Using inequalities \eqref{estJ1} and \eqref{estJ2} we see that all the terms of the form $P_nJ(\psi_n,\Delta\psi_n)$ are uniformly bounded in $L^2_tD(A)^*$.

Therefore, for a generic constant $ K >0$ we obtain that $\partial_t  (\Delta-K)\psi_n^c$ is uniformly bounded in $L^{5/4}_tD(A)^*_x$, being a subset of both $L^{5/4}_tH^{-1}_x$ and $L^{2}_tD(A)^*_x$, and by Aubin-Lions Lemma we get that $\psi^c_n$ converges strongly in $L^2_t H^1_{0,x}$. 
The other streamfunctions can be treated in the same way and we get that $\partial_t  (\Delta-K)\psi_{o,n}$ and $\partial_t \Delta \psi_n^t$ are uniformly bounded in $L^{2}_tD(A)^*_x$.

Now let us briefly discuss the convergence of the various terms in the equations for MAOOAM \ref{bb} and \ref{Temp}.
Starting with the nonlinear terms representing the infrared radiation we have to show that $P_n |T_{n, i}|^3T_{n,i} \rightharpoonup |T_i|^3T_i$, $i = a,o$. It is sufficient (see e.g. Lemma~8.3 in \cite{Robinson}) to know that $T_{o,n}$ is bounded in $L^{5/4}_tL^{5/4}_x$ and $T_{a,n}$ is bounded in $L^{2}_tL^{2}_x$ and that $|T_{n, i}|^3T_{n,i}$ converges to $ |T_i|^3T_i$, $i= a,o$, almost everywhere in~$[0, t_*]\times\Omega$.
This fact is a consequence of the more general property that, taken a generic function $f$ and a sequence $u_n$, if $f(u_n)$ is uniformly bounded in $L^q_tH^{-1}_x$ and $f(u_n)\rightharpoonup \chi$ in $L^q_tH^{-1}_x$ for some $1<q<\infty$ then $P_nf(u_n)\rightharpoonup \chi$ in the same space. 

For the short-wave radiation terms $R_a(T_a)$ and $R_o(T_o)$ we know they are continuous with values in $L^2_tL^2_x$ and $L^{5/4}_tL^{5/4}_x$, and for the Jacobian $J(\psi_{o,n},T_{o,n})$ it can be shown to converge in $C([0, t_*];D(A))^*$ using standard techniques.

Next, although $\pdv{T_{o,n}}{t}$ converges to $\pdv{T_o}{t}$ only in $L^{5/4}_t H^{-1}_x$, we observe that the limiting function is actually more regular, namely $\pdv{T_o}{t}\in L^{5/4}_tL^{5/4}_x+L^2_t H^{-1}_x$, because we proved that $|T_{o}|^3T_{o}$, $R_o(T_{o})$ and $ J(\psi_o,T_o)$ are in $ L^{5/4}_t L^{5/4}_x$ and $\Delta T_o\in L^2_tH^{-1}_x$.
Together with the regularity of $T_o\in L^5_tL^5_x\cap L^2_tH^1_x$ this implies that $T_o\in C([0, t_*]; L^2(\Omega))$ (it can be showed applying the result presented for example in \cite[Exercise~8.2]{Robinson}).

Moving to the streamfunctions equations, using standard techniques one can also show that $J(\psi_{a,n}^t, T_{a,n})$ converges weakly in $L^2_tD(A)^*$ and $J(\psi_n,\Delta\psi_n)$ converges weakly in $C([0, t_*],D(A))^*$. Moreover we have 
$$\pdv{\Delta\psi^t}{t},\ \pdv{\Delta\psi_o}{t}\in L^2_tD(A)^*_x \quad \text{and} \quad \pdv{\Delta\psi^c}{t}\in L^{5/4}_tL^{5/4}_x+L^2_tD(A)^*_x.$$

Using a variant of Theorem~7.2 (see \cite{mythesis}, Proposition 3.1) combined with Exercise~8.2, both in \cite{Robinson}, one can show that $\bm{\psi}\in C([0, t_*]; \bm{H}_0^1)$.

Last, one can prove that the limiting solution fulfils the initial condition proceeding as e.g. in \cite[Section 7.4.4]{Robinson}, and that the energy inequality holds for the limiting function following for example \cite[Theorem 4.6]{RobinsonNS} (again we refer to \cite{mythesis}, Proposition 3.2, for the details about the proof of the latter result). Thus we have shown the existence of a \emph{Leray-Hopf} weak solution for MAOOAM.

\subsection{Uniqueness of Weak Solutions}\label{uniq}
Let us start by considering the difference between two weak solutions $\delta\psi=\tilde{\psi}-\hat{\psi}$ and $\delta T=\tilde{T}-\hat{T}$. We know that the equations for $\delta\psi_a^t$ and $\delta\psi_o$ hold in $L^2_tD(A)^*_x$, while the one for $\delta\psi_a^c$ is defined in $L^{5/4}_tL^{5/4}_x+L^2_tD(A)^*_x$. Recalling  that
$\delta \psi_a^c \in  L^5_tL^5_x\cap L^2_tD(A)_x$ and $\delta\psi_a^t$ and $\delta\psi_o$ in $L^2_tD(A)_x$ 
we can apply Proposition 3.1 in \cite{mythesis} and hence, taking the scalar product of these equations with $-\bm{\delta\psi} = -\left( \delta\psi_a^t, \delta \psi_a^c, \delta\psi_o\right)$ in $\bmL^2$ with weight $(1,1,\kappa)$ (we recall that $\kappa=\frac{\rho h k_d}{2C}$),
we get the following expression:
\begin{multline}
\frac{1}{2}\pdv{}{t}\bl(\Vert\bm{\delta\psi}\Vert_{\bm{H}^1}^2+\kappa\frac{|\delta\psi_o|^2}{L^2_R}\br)-(J(\hat{\psi_a^t},\delta\psi_a^t),\Delta\delta\psi_a^t)-(J(\hat{\psi_a^c},\delta\psi_a^c),\Delta\delta\psi_a^t)-(J(\hat{\psi_a^c},\delta\psi_a^t),\Delta\delta\psi_a^c) \notag\\
-(J(\hat{\psi_a^t},\delta\psi_a^c),\Delta\delta\psi_a^c)-\kappa (J(\hat{\psi_o},\delta\psi_o),\Delta\delta\psi_o)+\frac{k_d}{2}|\nab (\delta\psi_a^t+\delta\psi_a^c-\delta\psi_o)|^2 \label{eq:unique:1}\\
+\nu_S\Vert\bm{\delta\psi}\Vert_{\bm{H}^2}^2
+2k'_d|\nab \delta\psi_a^c|^2+r\kappa|\nab\delta\psi_o|^2=\frac{f_0}{p_{\delta}}(\tilde{\omega}-\hat{\omega},\delta\psi_a^c),\notag
\end{multline}
where we have used 
$(J(\delta\psi^c_a,\tilde{\psi}^c),\Delta\delta\psi^t)=-(J(\delta\psi^t_a,\tilde{\psi}^c),\Delta\delta\psi^c)$ and $J(\delta\psi,\delta\psi)=0$; 
the remaining terms in $J$ can be estimated using \eqref{estJ1}, \eqref{estJ2} and Young's inequality by
\[ 
    \frac{\nu_S}{2}\Vert\bm{\delta\psi}\Vert_{\bm{H}^2}^2+C_1\Vert\bm{\hat{\psi}}\Vert^2_{\bmH^2}\Vert\bm{\delta\psi}\Vert_{\bm{H}^1}^2.
\]
It remains to find an expression for $(\tilde{\omega}-\hat{\omega},\delta\psi_a^c)$. 

We start observing that, thanks to the regularity of the equation for $\psi^c$, $\tilde{\omega},\hat{\omega}\in L^{5/4}_t L^{5/4}_x+ L^2_tD(A)^*_x$, and since $\delta\psi_a^c\in L^5_tL^5_x\cap L^2_tD(A)_x$ the scalar product is well defined. Now recalling that $\delta T_a=\frac{2f_0}{R}\delta\psi_a^c$ by \ref{eq:Tb}, we can use this relation to get another expression for $(\tilde{\omega}-\hat{\omega},\delta\psi_a^c)$ and eventually eliminate this term from the two equations. 

To close the system we also have to take into account the equation for $\delta T_o$, which is well defined in $L^{5/4}_t L^{5/4}_x+ L^2_tH^{-1}_x$ and thus we can take the inner product with $\delta T_o\in L^5_tL^5_x\cap L^2_t H_{0,x}^1$. We observe that $\delta T_a$ and $\delta T_o$ both have Dirichlet boundary conditions because $\tilde{T}^{(0)}=\hat{T}^{(0)}$. 
Then we find (using the notation $|\cdot|_{\bbR}$ to denote the Euclidean norm in order to avoid ambiguity with the $L^2$ norm $|\cdot|$):
\begin{multline}\label{eq:unique:2}
    \frac{f_0}{p_{\delta}}(\tilde{\omega}-\hat{\omega},\delta\psi_a^c)=\mu\bl(-\frac{1}{2}\pdv{}{t}\{\gamma_a |\delta T_a|^2+\gamma_o|\delta T_o|^2\}-\tilde{\nu}_T(|\nab \delta T_a|^2+|\nab \delta T_o|^2)\\
    -\gamma_a(J(\delta\psi_a^t,\tilde{T}_a),\delta T_a)-\gamma_o(J(\delta\psi_o,\tilde{T}_o),\delta T_o)-\lambda|\delta T_a-\delta T_o|^2+\epsilon_a\sigma_B(|\tilde{T}_o|_{\bbR}^3\tilde{T}_o\\
    -|\hat{T}_o|_{\bbR}^3\hat{T}_o,\delta T_a)-2\epsilon_a\sigma_B(|\tilde{T}_a|^3_{\bbR}\tilde{T}_a-|\hat{T}_a|^3_{\bbR}\hat{T}_a,\delta T_a)-\sigma_B(|\tilde{T}_o|^3_{\bbR}\tilde{T}_o-|\hat{T}_o|^3_{\bbR}\hat{T}_o,\delta T_o)\\
    +\epsilon_a\sigma_B(|\tilde{T}_a|^3_{\bbR}\tilde{T}_a-|\hat{T}_a|^3_{\bbR}\hat{T}_a,\delta T_o)+(R_a(\tilde{T}_a)-R_a(\hat{T}_a),\delta T_a)+(R_o(\tilde{T}_o)-R_o(\hat{T}_o),\delta T_o) \br).
\end{multline}
As $(|\tilde{T}_a|_{\bbR}^3\tilde{T}_a-|\hat{T}_a|^3_{\bbR}\hat{T}_a,\delta T_a)\ge 0$, $(|\tilde{T}_o|^3_{\bbR}\tilde{T}_o-|\hat{T}_o|^3_{\bbR}\hat{T}_o,\delta T_o)\ge 0$ these terms can be handled with ease. 
In order to control $(|\tilde{T}_o|^3_{\bbR}\tilde{T}_o-|\hat{T}_o|_{\bbR}^3\hat{T}_o,\delta T_a)$ we will use the $L^p_tL^q_x$-interpolation.
\begin{rem}
    The estimate for the mixed quartic term $(|\tilde{T}_o|_{\bbR}^3\tilde{T}_o-|\hat{T}_o|_{\bbR}^3\hat{T}_o,\delta T_a)$ is not straightforward, because we cannot apply the standard techniques for parabolic reaction-diffusion equations, see e.g. \cite[Chapter 8]{Robinson}, \cite[Section 1.1.4]{Temam}, or \cite[Section 11.B]{Smoller}.
    The main idea is that we shift some of the regularity on $\delta T_a$ exploiting the higher regularity of $\delta T_o$.
\end{rem}

We can rewrite this terms as follows:
\begin{align}
\label{sviluppo}
    (|\tilde{T}_o|^3_{\bbR}\tilde{T}_o-|\hat{T}_o|^3_{\bbR}\hat{T}_o,\delta T_a)&=(|\hat{T}_o|^3_{\bbR}\delta T_o+(|\tilde{T}_o|^3_{\bbR}-|\hat{T}_o|^3_{\bbR})\tilde{T}_o,\delta T_a)\notag\\
    &\le(|\hat{T}_o|^3_{\bbR}\delta T_a,\delta T_o)+2((\tilde{T}_o^2+\hat{T}_o^2)\tilde{T}_o\delta T_a,|\delta T_o|_{\bbR}).
\end{align}
The two terms in the right-hand side of \eqref{sviluppo} have the same degree of nonlinearity, so we will only present the computations in details for the first term. Choosing \(p'\) and \(q'\) such that \(\frac{1}{p'}+\frac{1}{q'}=\frac{1}{2}\) we get
\begin{align}
    \label{Tpq}
    (|\hat{T}_o|_{\bbR}^3\delta T_a,\delta T_o)\le |\hat{T}_o|_{L^{3p'}}^3|\delta T_a|_{L^{q'}}|\delta T_o|.
\end{align}
Choosing for example \(3p'=8\) we get \(\frac{1}{q'}=\frac{1}{2}-\frac{3}{8}=\frac{1}{8}\) and we can write \eqref{Tpq} as follows:
\begin{align}
    \label{stimaTo}
    |\hat{T}_o|_{L^8}^3|\delta T_a|_{L^8}|\delta T_o|\le C_2(|\hat{T}_o|_{L^8}^3|\nab\delta T_a|^2+|\hat{T}_o|_{L^8}^3|\delta T_o|^2)
\end{align}
where we have used the embedding of $H^1$ in $L^8$ for $\delta T_a$ and Young's inequality. 
Now the right-hand side of  \eqref{stimaTo} is well defined because $\delta T_a\in L^{\infty}_t H^1_x$, $\delta T_o\in L^{\infty}_tL^2_x$ and we know from $L^p_tL^q_x$ interpolation that $\hat{T}_o\in L^3_t L^8_x$.

Looking back at \eqref{eq:unique:2} we now want to estimate \((|\tilde{T}_a|_{\bbR}^3\tilde{T}_a-|\hat{T}_a|^3_{\bbR}\hat{T}_a,\delta T_o)\). We can use a similar strategy, and even easier here because \(\delta T_a,\ \hat{T}_a,\ \tilde{T}_a\) are elements of \(L_t^{\infty}H^1_x\). In fact we have
\begin{align}
    \label{rrrr}
    (|\hat{T}_a|_{\bbR}^3\delta T_a,\delta T_o)\le |\hat{T}_a|_{L^8}^3|\delta T_a|_{L^8}|\delta T_o|\le C_3(|\nab\hat{T}_a|^6|\nab\delta T_a|^2+|\delta T_o|^2),
\end{align}
using again the embedding of $H^1$ in $L^8$ for $\delta T_a,\ \hat{T}_a$ and Young's inequality. 
The other terms in $(|\tilde{T}_a|_{\bbR}^3\tilde{T}_a-|\hat{T}_a|_{\bbR}^3\hat{T}_a,\delta T_o)$ in \eqref{eq:unique:2} will be estimated in the same way.

Next, the estimates for the short-wave radiation terms in \eqref{eq:unique:2} are directly given by the assumption \eqref{xjgpw}.

Finally, we estimate the Jacobian terms with \eqref{estJ2} to arrive at the following inequality:
\begin{align*}
    &\pdv{}{t}\biggl( \Vert\bm{\delta\psi}\Vert_{\bm{H}^1}^2+\kappa\frac{|\delta\psi_o|^2}{L^2_R}+\mu(\gamma_a|\delta T_a|^2+ \gamma_o|\delta T_o|^2)\biggr)\le C\bl(\Vert\bm{\hat{\psi}}\Vert^2_{H^2}+\Vert\hat{T}_a\Vert^6_{H^1}+\Vert\tilde{T}_a\Vert^6_{H^1}\\
    &+|\hat{T}_o|_{L^8}^3+|\tilde{T}_o|_{L^8}^3+\Vert\tilde{T}_o\Vert^2_{H^1}+\Vert\tilde{T}_a\Vert^2_{H^2}+\mathcal{L}_R(t)\br)\bl(\Vert\bm{\delta\psi}\Vert_{\bm{H}^1}^2+\kappa\frac{|\delta\psi_o|^2}{L^2_R}+\mu(\gamma_a|\delta T_a|^2+ \gamma_o|\delta T_o|^2)\br).
\end{align*}
To have a more compact formula we can define the norm \begin{align}
\label{weaknorm}
\Vert (\bm\psi, T_o)\Vert_{\bmH_0^1\times L^2}^2:=\Vert \bm{\psi}\Vert_{\bmH^1}^2+\kappa\frac{|\psi_o|^2}{L^2_R}+\mu(\gamma_a|T_a-T_a^{(0)}|^2+ \gamma_o|T_o-T_o^{(0)}|^2),
\end{align} and we denote the integrable in time prefactor of $\Vert (\bm{\delta\psi}, \delta T_o)\Vert_{\bmH_0^1\times L^2}^2$ by $B_{\bbW}$.
Applying Gronwall's Lemma we find 
\[ 
    \Vert (\bm{\delta\psi}, \delta T_o)(t)\Vert^2_{\bmH_0^1\times L^2}\le e^{\int_0^t B_{\bbW}(s)ds}\Vert (\bm{\delta\psi}, \delta T_o)(0)\Vert^2_{\bmH_0^1\times L^2},
\]
ensuring continuity with respect to the initial conditions and uniqueness of the weak solution as desired.

\subsection{Continuity with respect to radiation parameters}
\label{contradparam}
In this subsection we show that the weak solution depends continuously on all the parameters which determine the intensity of the radiation in our model: the coefficients $\epsilon_a$, $\lambda$ and the short-wave radiation functions $R_a$ and $R_o$. The value of the emissivity coefficient $\epsilon_a$ reflects the level of greenhouse gases in the atmosphere.
In our model, a change on the emissivity of the atmosphere affects directly the behaviour of $T_a$ and $T_o$ and, as a consequence of the coupling, also of the other streamfunctions $\psi^t$ and $\psi_o$. So it can be interesting to understand in which sense our physical variables depend on $\epsilon_a$: we show that a small variation on the emissivity produces a small variation on the behaviour of the solution of MAOOAM. This will be true also for $\lambda\in \bbR_+$, which represents the heat transfer coefficient at the ocean–atmosphere interface and describes the effect of all non-radiative vertical exchanges of energy between the geophysical fluids, and for the functions
$R_i(T_i)$, $i=a,o$, describing the effect of short-wave radiation, in the norm defined as 
\begin{align}
        \label{normaRsup}
        \Vert R_i\Vert=\sup_T\frac{|R_i(T)|_{L^2}}{1+|T|_{L^4}}, \quad i=a,o.
    \end{align}  
We can define the vector $\bm{R}(\bm{T}):=(R_a(T_a),R_o(T_o))$ and the norm $\Vert \bm{R}\Vert^2=\Vert R_a\Vert^2+\Vert R_o\Vert^2$. We have the following:
\begin{thm}
\label{teoregpar}
    Let $\epsilon_a\in (0,1]$, $\lambda\ge 0$. Consider two radiation functions $\bm{R}(\bm{T})$ and $\tilde{\bm{R}}(\tilde{\bm{T}})$ with $\Vert \bm{R}(s)\Vert$ and $\Vert \tilde{\bm{R}}(s)\Vert$ square integrable in time 
    and such that the following inequality holds:
    \begin{align}
        \label{stimaR}
        (R_i(T_i)-\tilde{R}_i(\tilde{T}_i),\delta T_i)\le \frac{1}{2}\Vert R_i-\tilde{R}_i\Vert^2+c\big(1+|T_i|_{L^4}^2+\mathcal{L}_{R_i}(t)\big)|\delta T_i|^2+\frac{\nu_T}{2}|\nab\delta T_i|^2,
    \end{align}
    where $i=a,o$ and  $\mathcal{L}_{\hat{R}_i}(t)$ is a locally in time integrable function.
    Let $(\bm\psi, T_o)(t,\epsilon_a,\lambda,\bm{R})$ and $(\bm\psi,T_o)(t,\tilde{\epsilon}_a,\tilde{\lambda},\tilde{\bm{R}})$ be two weak solutions of MAOOAM. 
    Then for any $\tilde{\epsilon}_a,\ \tilde{\lambda},\ \tilde{\bm{R}}$ there exists a positive $C$ (which depends continuously on $ t,\epsilon_a,\lambda$) such that \begin{align*}
    \Vert(\bm\psi, T_o)&(t,\epsilon_a,\lambda,\bm{R})-(\bm\psi,T_o)(t,\tilde{\epsilon}_a,\tilde{\lambda},\bm{\tilde{R}})\Vert^2_{\bmH_0^1\times L^2} \\
    &\le C\bl( (\epsilon_a-\tilde{\epsilon}_a)^2+(\lambda-\tilde{\lambda})^2+\int_0^t\Vert \bm{R}(s)-\bm{\tilde{R}}(s)\Vert^2ds\br),\quad \text{for all }\ t\in [0, t_*],
    \end{align*}
    that is, on $[0, t_*]$ the solution of MAOOAM depends in a Lipschitz continuous way on $\epsilon_a,\ \lambda$ and $\bm{R}$. 
\end{thm}

\begin{proof}
We start with the emissivity coefficient $\epsilon_a$. We can proceed as we did for the uniqueness, with the only difference that now we consider two different coefficients \(\tilde{\epsilon}_a\) (emissivity related to \(\tilde{T}_a\)) and \(\hat{\epsilon}_a\) (emissivity related to \(\hat{T}_a\)). Therefore in the equation for \(\delta T_a=\tilde{T}_a-\hat{T}_a\) the only new terms to estimate are
\begin{align*}
    -2\tilde{\epsilon}_a\sigma_B|\tilde{T}_a|^3_{\bbR}\tilde{T}_a 
    + 2\hat{\epsilon}_a\sigma_B|\hat{T}_a|^3_{\bbR}\hat{T}_a
    + \tilde{\epsilon}_a\sigma_B|\tilde{T}_o|^3_{\bbR}\tilde{T}_o
    -\hat{\epsilon}_a\sigma_B|\hat{T}_o|^3_{\bbR}\hat{T}_o.
\end{align*}
Multiplying with \(\delta T_a\) we can rewrite the first two terms of the above equation as 
\[
-2\tilde{\epsilon}_a\sigma_B(|\tilde{T}_a|^3_{\bbR}\tilde{T}_a-|\hat{T}_a|^3_{\bbR}\hat{T}_a, \delta T_a)-2(\tilde{\epsilon}_a-\hat{\epsilon}_a)\sigma_B(|\hat{T}_a|^3_{\bbR}\hat{T}_a,\delta T_a).
\]
The first one is clearly negative, so again we can neglect it, while the second can be estimated easily (for example) as follows:
\[
2(\tilde{\epsilon}_a-\hat{\epsilon}_a)\sigma_B(|\hat{T}_a|^3_{\bbR}\hat{T}_a,\delta T_a)\le (\tilde{\epsilon}_a-\hat{\epsilon}_a)^2+\sigma_B^2|\hat{T}_a|_{L^8}^8|\delta T_a|^2
\]
where we used Cauchy-Schwarz and Young's inequalities.
The third and the fourth term have to be treated with extra care because \(\delta T_o\) has less regularity. We rewrite them as
\begin{equation}\label{eq:weak:cont1}
        \tilde{\epsilon}_a\sigma_B(|\tilde{T}_o|^3_{\bbR}\tilde{T}_o-|\hat{T}_o|^3_{\bbR}\hat{T}_o, \delta T_a)+(\tilde{\epsilon}_a-\hat{\epsilon}_a)\sigma_B(|\hat{T}_o|^3_{\bbR}\hat{T}_o,\delta T_a).
\end{equation}
The first term in \eqref{eq:weak:cont1} is the most singular but can be controlled using  \eqref{sviluppo} to \eqref{stimaTo} in the Section~\ref{uniq}, while for the second we use \(\hat{T}_o\in L^5_tL^5_x\) combined with
Hölder's inequality $(5/4,5)$ and the embedding $H^1\subset L^5$ to conclude
\begin{align*}
    (\tilde{\epsilon}_a-\hat{\epsilon}_a)\sigma_B(|\hat{T}_o|^3_{\bbR}\hat{T}_o,\delta T_a)\le \frac{(\tilde{\epsilon}_a-\hat{\epsilon}_a)^2}{2}|\hat{T}_o|^4_{L^5}+\frac{C\sigma_B^2}{2}|\hat{T}_o|^4_{L^5}|\nab \delta T_a|^2,
\end{align*}
where $C$ is a positive constant which comes from the above embedding.

Let us now consider the radiative terms appearing in the equation for \(\delta T_o=\tilde{T}_o-\hat{T}_o\):
\begin{align}
    \label{mnmnmn}
    -\sigma_B(|\tilde{T}_o|^3_{\bbR}\tilde{T}_o-|\hat{T}_o|^3_{\bbR}\hat{T}_o)+\tilde{\epsilon}_a\sigma_B|\tilde{T}_a|^3_{\bbR}\tilde{T}_a-\hat{\epsilon}_a\sigma_B|\hat{T}_a|^3_{\bbR}\hat{T}_a.
\end{align}
When we multiply the above equation with \(\delta T_o\) the first term becomes negative, so again we can neglect it. 
The two remaining terms in \eqref{mnmnmn} multiplied with $\delta T_o$ can be estimated as we did in Section~\ref{uniq} using \eqref{rrrr}.

Summarizing, we find an inequality of the form
\begin{align*}
    \Vert (\bm{\psi},T_o)(t,\tilde{\epsilon}_a)-(\bm{\psi},T_o)(t,\hat{\epsilon}_a)\Vert^2_{\bmH_0^1\times L^2}&\le Ce^{k(t)}\Vert (\bm{\psi},T_o)(0,\tilde{\epsilon}_a)-(\bm{\psi},T_o)(0,\hat{\epsilon}_a)\Vert^2_{\bmH_0^1\times L^2}\\
    &+\frac{(\tilde{\epsilon}_a-\hat{\epsilon}_a)^2}{2}e^{k(t)}\bl(3t+\int_0^t |\hat{T}_o(s)|^4_{L^5}ds \br)
\end{align*}
where \(k(t)\) is a locally integrable function. Then \((\bm{\psi}^{\epsilon_a}(t),T_o^{\epsilon_a}(t))\) is Lipschitz continuous with respect to \(\epsilon_a\) for all \(t>0\) in the $\bmH_0^1\times L^2$ norm.

Using the same strategy it is straightforward to show that the solution is Lipschitz-continuous in the $\bmH_0^1\times L^2$ norm also with respect to the parameter $\lambda$.

It is left to study the continuity of the solution with respect to the short-wave radiation function $R_i(t,x,y,T)$, $i=a,o$. 
Taking the norm defined in \eqref{normaRsup} and the assumption \eqref{stimaR}, following the same strategy again we obtain Lipschitz-continuity.
\end{proof}

\begin{rem}
We want to exhibit two examples of radiation functions which
fullfil all the assertions of Theorem~\ref{teoregpar} and one can give an explicit bound for the norm $\Vert \cdot \Vert $ given by \eqref{normaRsup} so that the right-hand side of the continuity estimate becomes explicit. 

First, we assume that the radiation function is of the form $R(T)=R^{(1)}(t,x,y)+R^{(2)}(t,x,y)T$ with $R^{(1)}\in L^2_tL^2_x$ and $R^{(2)}\in L^2_tL^4_x$, the norm \eqref{normaRsup} can be estimated by \begin{align*}
\Vert R \Vert\le |R^{(1)}|_{L^2}+|R^{(2)}|_{L^4},
\end{align*}
and also the square-integrability $\Vert \bm{R}(s)\Vert$ and $\Vert \tilde{\bm{R}}(s)\Vert$ follows.
Moreover, since radiation functions of this form are also Lipschitz continuous as functions from $L^4$ to $L^2$ with respect to the temperature, \eqref{stimaR} is fulfilled.

Second, we assume that the radiation functions depends on a parameter $\alpha$ with $|\partial_\alpha R(\alpha ,T)|_{L^2}\le c(t)(1+|T|_{L^4})$ and $R$ is Lipschitz from $L^4$ to $L^2$ in $T$. If both $c(t)$ and the Lipschitz constant are locally square integrable, then $\Vert R(\alpha_1,\cdot)-R(\alpha_2, \cdot)\Vert\le c(t) |\alpha_1-\alpha_2|$, and the square integrability of the norms, and inequality \eqref{stimaR} is fulfilled.  

\end{rem}

\begin{rem} We cannot show continuity of the solution with respect to the emissivity of the ocean (which in our model is simply \(\epsilon_o=1\), since the ocean is considered as a black body) in a similar manner, because in this case the analogue of the first term in \eqref{mnmnmn} would be not automatically non-negative and we do not  know how to estimate it, as we do not know whether $T_o$ belong to $L^8_t L^5_x$.
\end{rem}
\section{Solutions with Higher Regularity}
In this section we want to show that assuming higher regularity on the initial datum we can find solutions for MAOOAM with the same regularity of the initial conditions. 

We recall some results about the fractional powers of the operator $A$. Using spectral theory we define the domain $D(A^{\frac{s}{2}})$, $s\in \mathbb{N}$, as
\begin{align}
\label{specdec}
D(A^{\frac{s}{2}})=\bl\{ u=\sum_{j=1}^{\infty} (u,w_j)w_j\ : \quad \sum_{j=1}^{\infty} \lambda^{s}_j |(u,w_j)|^2<\infty\br\},
\end{align}
where $w_j$ and $\lambda_j$ are respectively the eigenfunctions and the eigenvalues of $A$. Since the operator $A$ is positive defined and Poincaré inequality holds, all $\lambda_j$ are positive.
If $\Omega$ is a $C^k$ domain or a rectangle, it can be shown that definition \eqref{specdec} is equivalent for $ s\leq k$ to the following: 
\begin{align}
    \label{defDAalter}
    D(A^{s})=\{u\in H^{2s}(\Omega):\quad \restr{\Delta^ju}{\partial\Omega}=0,\ \text{ for all }\ j\in \bbN_0,\ j< s\}
\end{align}
and for $u\in D(A^{s/2})\subset H^s(\Omega)$ it can be shown that 
\begin{align}
    \label{equiAH}
    |A^{s/2}u|\le \Vert u\Vert_{H^s}\le c_s|A^{s/2}u|
\end{align}
for some constant $c_s$.
The above results are  well known for domains with sufficient smooth boundary; in the Appendix \ref{rectdom} we will discuss the case with $\Omega$ as a rectangle (we recall that the domain we use for the ocean has this shape).

\subsection{Quasi-Strong Solution}\label{QSsec}
In the discussion about the Weak Solution we faced the asymmetry in the regularity of $\bm{\psi}$ and $T_o$. A natural question arise about the existence of a solution in which the ocean temperature $T_o$ has the same regularity of the streamfunction $\bm{\psi}$. This leads us to the definition of a different class of solutions:
\begin{defm}[Quasi-Strong Solution for MAOOAM]
    We say that a function $(\bm{\psi},T_o)$ is a Quasi-Strong Solution for the system \ref{bb}, \ref{Temp}, \ref{eq:Tb} if it is a Weak Solution and in addition $(\bm{\psi},T_o-T_o^{(0)})\in \bbL^{\infty}_t\bbD(A^{\frac{1}{2}})\cap \bbL^2_t\bbD(A)$.
\end{defm}

Note that the estimate \eqref{equiAH} implies that a Quasi-Strong Solution is such that $(\bm{\psi},T_o)\in \bbL^{\infty}_t\bbH^1_x\cap \bbL^2_t\bbH^2_x$. Recall that for the existence of Weak Solutions we have already shown $\bm\psi\in \bmL^{\infty}_t\bmH^1_{0,x}\cap \bmL^2_t\bmH^2_x$, so what remains to be established is the additional regularity for $T_o$ without requiring higher regularity for $\bm\psi$. 

We will require that the short-wave radiations are functions only dependent on the temperature in order to later ensure that the boundary conditions \eqref{Tobal} arising from the higher regularity we are seeking for $T_o$ are fulfilled. In addition the ocean function $R_o$ should satisfy the following estimates:
\begin{assumption}\label{assum:qss}
     $R_o$ is a $C^1$ function such that  
\begin{enumerate}
    \item there exists $c_1, c_2>0$ such that
    \begin{align}
        \left|(\nab R_o(g),\nab g) \right|&\le c_1|g|^2+ c_2|\nab g|^2, \quad g \in L^2_tH^1_x. \label{eq:condgradR1}
    \end{align}
    \item there exists a positive constant $\mathcal{L}_{R_o}$ such that
    \begin{equation}
        \left(R_o(\hg)-R_o(\tg),\Delta(\hg- \tg)\right) \le \mathcal{L}_{R_o}\left|\nab(\hg-\tg)\right|^2 +\frac{\nu_T}{8}\left|\Delta (\hg - \tg)\right|^2 \label{stimaqsR}
    \end{equation}
    for all  $\hg, \tg\in L^2_tD(A)$.
    \end{enumerate} 
\end{assumption}
We can then show existence and uniqueness of quasi-strong solutions and their continuity with respect to initial conditions:

\begin{thm}
    \label{teoexuniQS}
    Let the hypotheses of Theorem~\ref{teoexuniqweak} be fulfilled. Assume 
    \begin{enumerate}
        \item $R_a$ and $R_o$ depend only on the temperature, i.e. $R_i(t,x,y,T(t,x,y)) = R_i(T)$, $i = a,o$;
        \item $R_o$ satisfies the bounds in \cref{assum:qss};
        \item the reference parameters $T_a^{(0)}$, $T_o^{(0)}$ are the stationary and homogeneous solutions of \ref{Temp} with such $R_a$, $R_o$, that is,
        \begin{align}
        \label{stazomT}
    \begin{pmatrix}
    -\lambda -2\epsilon_a\sigma_B |T_a^{(0)}|^3 & \lambda +\epsilon_a\sigma_B |T_o^{(0)}|^3\\
    \lambda +\epsilon_a \sigma_B |T_a^{(0)}|^3 & -\lambda - \sigma_B |T_o^{(0)}|^3
    \end{pmatrix}
    \begin{pmatrix}
    T^{(0)}_a\\
    T^{(0)}_o
    \end{pmatrix}+
    \begin{pmatrix}
    R_a(T_a^{(0)})\\
    R_o(T_o^{(0)})
    \end{pmatrix}
    =\begin{pmatrix}
    0\\
    0
    \end{pmatrix}.
    \end{align}
    \end{enumerate}
    Then, for initial conditions $(\bm{\psi}(0),T_o(0)-T_o^{(0)})\in \bbD(A^{\frac{1}{2}})$, there exists a unique Quasi-Strong Solution \[(\bm{\psi},T_o-T_o^{(0)})\in C([0, t_*],\bbD(A^{\frac{1}{2}}))\cap \bbL^2(0, t_*; \bbD(A))\] for \ref{bb},\ref{Temp}, \ref{eq:Tb} which depends continuously on the initial datum.
\end{thm}

Using similar arguments to the ones we described in Section~\ref{contradparam} one can also show that the Quasi-Strong Solution is Lipschitz-continuous with respect to the parameters $\lambda$, $\epsilon_a$ and to the solar radiation forcing $R(T)$. We will omit the details.

As before we outline only the main points of the proof of \cref{teoexuniQS} in the following sections.

\subsubsection{Existence of Quasi-Strong solutions}
We focus on the equation for $T_o$ as it is what differs with respect to the Weak Solution result. 
Considering the Galerkin approximation (we will omit the index $n$) we can take the scalar product in $L^2_x$ of the equation for $T_o$ with $-\Delta T_o$.
For the integration by parts required in the time-derivative term we use that the boundary value $\restr{T_o}{\partial\Omega}=T_o^{(0)}$ does not depend on time and in order to treat the radiation terms we need the following boundary term to vanish
\begin{align*}
    \int_{\partial\Omega} \big(\lambda(T_a-T_o)-\sigma_B|T_o|^3_{\bbR}T_o+\epsilon_a\sigma_B|T_a|^3_{\bbR}T_a+R_o(T_o)\big)\nab T_o\cdot \hat{n}\ dl=0.
\end{align*}
At the boundary the ocean and atmosphere temperatures $T_o$ and $T_a$ take the constant reference values $T_o^{(0)}$ and $T_a^{(0)}$ respectively. Hence in order to satisfy the above equality it is sufficient to assume that
\begin{align}
    \label{Tobal}
    \lambda(T_a^{(0)}-T_o^{(0)})-\sigma_B|T_o^{(0)}|^3_{\bbR}T_o^{(0)}+\epsilon_a\sigma_B|T_a^{(0)}|^3_{\bbR}T_a^{(0)}+R_o(T_o^{(0)})=0 \quad \text{on}\ \partial\Omega.
\end{align}
Since in Theorem \ref{teoexuniQS} we assumed $T_a^{(0)}$ and $T_o^{(0)}$ to be stationary and homogeneous solutions for \ref{Temp}, they are solutions of \eqref{stazomT}, hence in particular satisfy \eqref{Tobal}. Next, we can control the term $(R_o(T_o),\Delta T_o)$ thanks to \eqref{eq:condgradR1}
\begin{align}
    \label{condgradR}
    -(\nab R_o(T_o),\nab T_o)\le c_1|T_o|^2+ c_2|\nab T_o|^2.
\end{align}
The transport term $(J(\psi_o,T_o),\Delta T_o)$ can be estimated using \eqref{estJ1}, while the infrared term $(|T_a|^3_{\bbR}T_a,\Delta T_o)$ can be easily estimated using 
H\"older and Young inequality and $(|T_o|^3_{\bbR}T_o,\Delta T_o)$ gives rise to
\begin{align*}
    -(4|T_o|^3_{\bbR}\nabla T_o,\nabla T_o)\le 0,
\end{align*}
hence we can neglect it.

\begin{rem}
    \label{postermTo}
    The negativity of the long-wave radiation term $(|T_o|^3_{\bbR}T_o,\Delta T_o)$ is essential. Indeed, if we try to estimate $\left|(4|T_o|^3_{\bbR}\nabla T_o,\nabla T_o)\right|$ we would need that $T_o$  in $L^6_tL^6_x$, but we only have, from the energy inequality \ref{enes1}, $L^5_tL^5_x$.
\end{rem}

Summarizing, we obtain the following inequality:
\begin{equation}\label{estTo}
\begin{split}
    \pdv{}{t} |\nabla T_o|^2+\frac{\tilde{\nu}_T}{\gamma_o}|\Delta T_o|^2\le C_1(\Vert \psi_o&\Vert^2_{H^2}+\Vert\psi_a^c\Vert^2_{H^2}+1)|\nabla T_o|^2 \\ &+C_2(\Vert\psi_a^c\Vert^6_{H^1}+\Vert\psi_a^c\Vert^2_{H^1}+|T_o|^2)+C_3,
\end{split}
\end{equation}
with $C_1$, $C_2$ and $C_3$ are positive constants.
Therefore applying Gronwall's inequality to \eqref{estTo} we can conclude that the Galerkin truncated function $T_{o,n}$ is uniformly bounded in $L^{\infty}_tH^1_x$ if $T_o(0)\in H^1_x$.
Integrating in time inequality \eqref{estTo} we find that $T_{o,n}$ is uniformly bounded in $L^2_tH^2_x$ as well.

To get a uniform bound for $\partial_t T_{o,n}$ in $L^2_tL^2_x$, it is sufficient to show that all the other terms in the equation for $T_o$ have this regularity: one can use the embedding of $H^1_x$ in $L^p_x$ to show that $|T_{o, n}|^3T_{o, n}$ is uniformly bounded in $L^{p}_t L^{q}_x$ for all $p,q\in[1,+\infty)$ and as $R_o$ is continuous and bounded in $L^2_tL^2_x$, $R_o(T_{o,n})$ is uniformly bounded in $L^2_tL^2_x$, then all the radiation terms have the desired regularity. 

Finally, the transport terms can be estimated using standard techniques and then we get that $\partial_t T_{o,n}$ is uniformly bounded in $L^2_tL^2_x$. Using Aubin-Lions Lemma and proceeding as we did for $T_a$ in Section~\ref{exweak} one can check that every term in the equation for $T_{o,n}$ converges to the corresponding limiting term.

In conclusion we have shown existence of the Quasi-Strong Solution and using \cite[Corollary 7.3]{Robinson}, we deduce that $T_o-T_o^{(0)}$ is time-continuous with values in  $D(A^{\frac{1}{2}})$.

\subsubsection{Continuous dependence on initial conditions}
\label{QSuniq}
Since a Quasi-Strong Solution is also a Weak Solution, it is unique by the results in Section~\ref{uniq}.
Now we want to show that a Quasi-Strong Solution of MAOOAM depends continuously on the initial datum $(\bm{\psi}(0),T_o(0)-T_o^{(0)})\in \bbD(A^{\frac{1}{2}})$.
The equations for the difference between two streamfunctions $\bm{\delta\psi}$ can be treated exactly in the same way we did for the uniqueness of the weak solution, so we will focus only on the equation for $\delta T_o=\tilde{T}_o-\hat{T}_o$.

We take the scalar product in $L^2_x$ of the equation for $\delta T_a$  with $\delta T_a$ as we did for the Weak Solution in Section~\ref{uniq}, while the equation for $\delta T_o$ is now multiplied in $L^2_x$ with $\Delta \delta T_o$, which is well defined as $\partial_t T_o$ and $\Delta \delta T_o$ are both in $L^2_tL^2_x$.
Let us consider the radiation terms one by one. The quantity $(|\tilde{T}_a|^3_{\bbR}\tilde{T}_a-|\hat{T}_a|^3_{\bbR}\hat{T}_a,\delta T_a)$ is positive, hence we can drop it. Next we can control $(|\tilde{T}_o|^3_{\bbR}\tilde{T}_o-|\hat{T}_o|^3_{\bbR}\hat{T}_o,\delta T_a)$ using H\"older and Poincaré inequalities and recalling that $\int_0^t |T_o(s)|^6_{L^{12}}ds$ is finite. Using integration by parts we can rewrite the next infrared radiation term as  
\begin{align}
    \label{cdfg}
    (|\tilde{T}_o|^3_{\bbR}\tilde{T}_o-|\hat{T}_o|^3_{\bbR}\hat{T}_o,\Delta\delta T_o)
    = -4(|\tilde{T}_o|^3_{\bbR}\nab \delta T_o,\nab\delta T_o)  -8((\tilde{T}_o^2+\hat{T}_o^2)\nab \hat{T}_o \delta T_o,\nab \delta T_o).
\end{align}
It is essential that the first term on the right hand side in \eqref{cdfg} is negative, as we pointed out in Remark~\ref{postermTo},
while the term in the second line can be estimated similarly as the previous infrared term (in this case we will need that $\int_0^t |\nab T_o(s)|^4|\Delta T_o(s)|^2ds$ is finite for all $t$). 

The last infrared radiation term $(|\tilde{T}_a|^3_{\bbR}\tilde{T}_a-|\hat{T}_a|^3_{\bbR}\hat{T}_a,\Delta\delta T_o)$ using again integration by parts can be decomposed in two summands as in \eqref{cdfg}. In this case the first summand has no sign, but we can anyway estimate both using again standard H\"{o}lder inequality with exponents $(4,4,2)$ and Young's inequality.
The part of the equation related to $R_o$ can be estimated thanks to \eqref{stimaqsR} in Assumption~\ref{assum:qss},
and the Jacobian terms related to the transport of $T_o$ can be controlled using H\"older and Ladyzhenskaya's inequalities and relation \eqref{estJ1}.

By the regularity of the Quasi-Strong Solution we have that $\Vert \tilde{T}_o(s)\Vert^2_{L^{\infty}_tH^1_x}$, $\int_0^t \Vert \tilde{T}_o(s)\Vert^2_{H^2}ds$ and $\int_0^t\Vert\hat{\psi}_o(s)\Vert^2_{H^2}ds$ are finite for all $t\in [0, t_*]$,
and then we get the estimate 
\begin{align*}
\pdv{}{t}\biggl( \Vert\bm{\delta\psi}\Vert_{\bm{H}^1}^2+&\kappa\frac{|\delta\psi_o|^2}{L^2_R}+\mu(\gamma_a|\delta T_a|^2+\gamma_o|\nab\delta T_o|^2)\biggr)\\
&\le B_{\bbQS}(t)\bl(\Vert\bm{\delta\psi}\Vert_{\bm{H}^1}^2+\kappa\frac{|\delta\psi_o|^2}{L^2_R}+\mu(\gamma_a|\delta T_a|^2+ \gamma_o|\nab\delta T_o|^2)\br)
\end{align*}
with an integrable function in time $B_{\bbQS}(t)$ collecting all the prefactors.

\subsection{Strong Solution}
\label{Strsec}
 For the Quasi-Strong Solution we asked for higher regularity on $T_o$. If we restore the asymmetry requiring higher regularity of $\bm{\psi}$ this leads to the following kind of solution:
\begin{defm}[Strong Solution for MAOOAM]
We say that a function $(\bm{\psi},T_o)$ is a Strong Solution for \ref{bb}, \ref{Temp}, \ref{eq:Tb} if it is a Weak Solution and $\bm{\psi}\in \bmL^{\infty}_t\bmD(A)\cap \bmL^2_t\bmD(A^{\frac{3}{2}}), \ T_o-T_o^{(0)}\in L^{\infty}_tD(A^{\frac{1}{2}})\cap L^2_tD(A)$.
\end{defm}

We have already treated the equation for $T_o$ for this kind of regularity in the previous subsection. Therefore it remains to find higher regularity estimates on the equations for the streamfunctions.
As usual we use the Galerkin approximation (omitting the index $n$) and we multiply the equations for $\bm{\psi}$ with $\Delta\bm{\psi}$ in $\bm{L}^2$ with weight $(1,1,\kappa)$. We also define $\Vert \bm{\psi}\Vert_{\bmH^3}^2:=|A^{\frac{3}{2}} \psi_a^t|^2+|A^{\frac{3}{2}} \psi_a^c|^2+\kappa|A^{\frac{3}{2}} \psi_o|^2$.

We observe that all the nonlinear terms related to the Jacobian operator vanish because $(J(\psi_i,\Delta\psi_j),\Delta\psi_k)=0$ when $j=k$ using the fundamental property \eqref{Jpropfond} and the fact that $\psi_i=0$ on the boundary, and when $j\neq k$ the various terms cancel each other. Recalling that $-(\omega,\Delta\psi_a^c)=(\omega,\frac{Rp_{\delta}}{pf_0}\Delta T_a)$ we can use the equation for $T_a$ to eliminate $\omega$, and all the terms coming from $(\omega,\Delta T_a)$ can be estimated as follows:
\begin{align}
\label{ineq: strongDeltaT_a}
\begin{split}
    &\int_0^t \big(\lambda(T_a-T_o)-\epsilon_a\sigma_B|T_o|^3_{\bbR}T_o+2\epsilon_a\sigma_B|T_a|^3_{\bbR}T_a-R_a(T_a)+J(\psi_a^t, T_a),\Delta T_a\big)ds \\
    &\quad \le C_2 \left( \Vert T_a-T_o\Vert^2_{L^2_tL^2_x}+\Vert T_o\Vert^8_{L^{\infty}_tH^1_x}+\Vert T_a\Vert^8_{L^{\infty}_tH^1_x} \right.\\
    &\qquad \left. +\Vert R_a(T_a)\Vert^2_{L^2_tL^2_x}+\Vert T_a\Vert^2_{L^{\infty}_tH^1_x}\Vert \psi_a^t\Vert^2_{L^2_tH^2_x}+\Vert T_a\Vert^2_{L^2_tH^2_x}\right).
    \end{split}
\end{align}
An interesting fact is that, due to the higher regularity of $T_a$ and differently to what we did to show existence of the Quasi-Strong Solution, to get inequality \eqref{ineq: strongDeltaT_a} we do not need integration by parts on the radiation terms. Indeed, the only integration by part is needed to estimate $(J(\psi_a^t, T_a),\Delta T_a\big)$ using \eqref{estJ1}, but this does not produce a boundary term because $\psi_a^t$ vanishes on the boundary. Hence, to show the existence of the Strong Solution we do not need an additional relation between the boundary values analogue to \eqref{Tobal}, and consequently the function $R_a$ can 
also explicitly depend on time and space.

Using the regularity in space of $\bm\psi$ one can also show that
$\pdv{\Delta\bm\psi}{t}\in \bmL^{2}_t\bmH^{-1}_x$ and hence $\bm{\psi}\in C([0, t_*],\bmH^2(\Omega))$. Strong Solutions, being also weak, are unique.


In order to prove that the Strong Solution generates a semidynamical system we need to show that the strong solutions depend continuously on the initial datum.
We use the  same technique we followed for the Quasi-Strong Solution in Section~\ref{QSuniq}.
We multiply the equations for $\bm{\delta\psi}=\bm{\tilde{\psi}}-\bm{\hat{\psi}}$, where $\bm{\tilde{\psi}},\ \bm{\hat{\psi}}$ are two Strong Solutions for \ref{bb}, by $\Delta\bm{\delta\psi}$ in $\bmL^2_x$ with weight $(1,1,\kappa)$. We note that $(J(\psi_i,\Delta\delta\psi_j),\Delta\delta\psi_k)=0$ when $j=k$ as a consequence of \eqref{Jpropfond} and the other Jacobian terms can be estimated using \eqref{estJ2}, while the boundary condition $\restr{A\bm{\delta\psi}}{\partial\Omega}=0$ allows to treat the viscosity term.
To get rid of the term with $\omega$, we take the $L^2_x$ scalar product of the equations for $\delta T_a=\tilde{T}_a-\hat{T}_a$ with $\Delta\delta T_a$ and of the equation for $\delta T_o=\tilde{T}_o-\hat{T}_o$ with $\Delta\delta T_o$.
All the radiative terms (linear, short-wave and infrared) and the Jacobian terms in this equation can be treated similarly as we did in the proof for the uniqueness of Quasi-Strong Solution, but in this case we need also to assume that $R_a$ satisfies the following
\begin{align}
    \label{altrastimaR}
    (R_a(\tilde{T}_a)-R_a(\hat{T}_a),\Delta\delta T_a)\le \mathcal{L}_{R_a}(t)|\nab\delta T_a|^2 +\frac{\nu_T}{8}|\Delta \delta T_a|^2,
\end{align}
where $\mathcal{L}_{R_a}(t)$ is locally integrable in time.

Applying Gronwall's Lemma we obtain a bound in terms of the following norm in $\bmH^2\times H^1$:
\begin{align*}
    \Vert (\bm\psi, T_o)\Vert^2_{\bmH^2\times H^1}:=&\Vert \bm{\psi}\Vert_{\bmH^2}^2+\kappa\frac{|\nab\psi_o|^2}{L^2_R}+\mu\bl(\frac{4\gamma_a f_0^2}{R^2}|\nab \psi_a^c|^2+ \gamma_o|\nab T_o|^2\br).
\end{align*}
Summarizing, we have proved the following: 
\begin{thm}[Existence of the Strong Solution for MAOOAM and continuity w.r.t. initial condition]
    Assume that the hypotheses of Theorem~\ref{teoexuniQS} are fulfilled, and that $R_a$ satisfies \eqref{altrastimaR}. If $(\bm{\psi}(0),T_o(0)-T_o^{(0)})\in \bmD(A)\times D(A^{\frac{1}{2}})$ then there exists a unique Strong solution 
    \[
        (\bm{\psi},T_o-T_o^{(0)})\in C\left([0, t_*],\,\bmD(A)\times D(A^{\frac{1}{2}})\right)
    \] 
    for MAOOAM which is Lipschitz continuous with respect to the initial datum in $\bmH^2\times H^1$.
\end{thm}
\section{Existence of the Global Attractor and Estimate of its Dimension}\label{chapactr}
In the previous sections we proved that all kind of solutions for the atmosphere-ocean model generate a semidynamical system, under the additional hypotheses that the short-wave radiations $R_a$ and $R_o$ may depend on the temperature but not on time. 
We will focus now on the class of weak solutions, so that our semidynamical system is $((\bmH_0^1 \times L^2)(\Omega),\ \{S(t)\}_{t\ge0})$, where $S(t)$ is the continuous semigroup of solutions operator $S(t)(\bm\psi(0),T_o(0))=(\bm\psi(t),T_o(t))$. We will show the existence of a finite dimensional global attractor and that the semigroup is injective on the attractor.
\subsection{Existence of the Attractor}
\label{existattr}
In this subsection, we show the existence of a compact absorbing set, which then by classical results (see e.g.~\cite[Theorem 10.5]{Robinson}) implies the following:
\begin{thm}
Assume that the hypotheses ensuring the existence and uniqueness of the Strong Solution are fulfilled. Additionally assume that there exist positive constants $c_1$ and $c_2$ such that  $|R_a(T_a)|^2\le c_1+c_2(|T_a|^2+|\nab T_a|^2)$. Then there exists a global attractor $\mathcal{A}$ for the solutions of MAOOAM in $(\bmH_0^1\times L^2)(\Omega)$.
\end{thm}
In the rest of the discussion we will use the following notation:
\begin{align}
\label{defWeS}
\bbW=\bmH_0^1\times L^2, \quad \bbS=\bmH^2\times H^1.
\end{align}
We start showing that there exists an absorbing set for $(\bm{\psi},T_o)$ in $\bbW$: in this space we will use the norm $\Vert\cdot\Vert_{\bmH_0^1\times L^2}$ defined in \eqref{weaknorm}. From the energy estimate \ref{enes1}, neglecting all the damping terms except the one related to the friction with the bottom of the ocean 
we get
\begin{multline}
\label{eness}
    \frac{1}{2}\pdv{}{t}\Vert (\bm\psi, T_o)\Vert_{\bmH_0^1\times L^2}^2 +\nu_S\Vert \bm{\psi}\Vert_{\bm{H}^2}^2+\mu\tilde{\nu}_T(|\nabla (T_a-T_a^{(0)})|^2+|\nabla (T_o-T_o^{(0)})|^2) +r\kappa |\nab \psi_o|^2\le E,
\end{multline} 
where $E$ is the same energy bound we derived in Section~\ref{sectenergy} but now independent of time.
From \eqref{eness}, using Poincaré inequality first and then Gronwall's Lemma we find
\begin{align*}
\Vert (\bm\psi, T_o)(t)\Vert_{\bmH_0^1\times L^2}^2\le \Vert (\bm\psi, T_o)(0)\Vert_{\bmH_0^1\times L^2}^2 e^{-\Lambda_0 t}+\frac{E}{\Lambda_0}(1-e^{-\Lambda_0 t}),
\end{align*}
where 
\[\Lambda_0=\lambda_1\min\bigg\{\nu_S,\frac{\tilde{\nu}_T}{\gamma_o},\frac{\tilde{\nu}_T}{\gamma_a}, rL^2_R\bigg\}\]
and $\lambda_1$ is the smallest eigenvalue of $-\Delta$.
Now we can define $\rho_{\bbW}^2:=\frac{2E}{\Lambda_0}$ such that 
\begin{align}
\label{ragASw}
\Vert (\bm\psi, T_o)(t)\Vert_{\bmH_0^1\times L^2}^2\le\rho_{\bbW}^2 , \ \quad \text{for all } t\ge \frac{1}{\Lambda_0}\ln(\frac{\Lambda_0\Vert (\bm\psi, T_o)(0)\Vert_{\bmH_0^1\times L^2}}{E})=:t_0,
\end{align}
so we can conclude that the solution $(\bm{\psi},T_o)$ of our model admits an absorbing set in $\bbW$. In addition, integrating on the interval $[t,t+1]$ inequality \eqref{eness} we can find the uniform bound in time $\int_t^{t+1}(\Vert \bm{\psi}(s)\Vert_{\bm{H}^2}+|\nab T_o(s)|^2)ds \le \frac{1}{2}(\rho_{\bbW}^2+2E)=:I_{\bbW}$ for $t\ge t_0$.

Using the values for the physical constants of the model  in \cite{MAOOAM} we find that $\Lambda_0=\lambda_1\frac{\tilde{\nu}_T}{\gamma_o}$ is of order $10^{-18}$. To get a lower bound for the order of magnitude of $E$ we drop all the contributions coming from $R_a$, $T_a^{(0)}$ and $T_o^{(0)}$ as they are negligible and we take $R_o$ constant with value $200\ W m^{-2}$ as in \cite{VannitsemGhil}. This tells us that $E$ is at least of order $10^{16}$.
\begin{rem}
    We observe that it is possible to find an absorbing set without using the diffusive terms. Hence, this will give a radius independent of the viscosity coefficients $\nu_S$ and $\tilde{\nu}_T$.  Indeed, ignoring in \ref{enes1}  the viscosity terms and using  instead 
    $k_d|\nab(\psi^t+\psi^c-\psi_o)|^2$, $2k'_d|\nab\psi^c|^2$, $r\kappa|\nab\psi_o|^2$, $\tilde{\epsilon}_a\sigma_B|T_a|^5_{L^5}$ and $\tilde{\epsilon}_a\sigma_B|T_o|^5_{L^5}$ we got an exponential decay rate of order $10^{-16}$, slightly smaller than the constant $\Lambda_0$ above.
\end{rem}

Next we want to find an absorbing set bounded  in $\bbS$: we start by computing a uniform bound in $H^1_x$ for $T_o$ for large times. We use inequality \eqref{estTo} and the estimates for the absorbing set of $\psi_a^c$ in $H^1_x$ and of $T_o$ in $L^2_x$ to get the following inequality:
\begin{align*}
\pdv{}{t} |\nabla T_o|^2\le\alpha_1(|\Delta\psi_o|^2+|\Delta\psi_a^c|^2+1)|\nab T_o|^2+\alpha_2, \quad \text{for all } t\ge t_0,
\end{align*}
where $t_0$ is the same of \eqref{ragASw}, $\alpha_1$ and $\alpha_2$ 
are positive constants which depend on the radii of the absorbing sets. Applying now the Uniform Gronwall's Lemma (see \cite[Chapter 1, Lemma 1.1]{Temam}) and the integral bound $\int_{t-1}^t (|\Delta\psi_o(r)|^2+|\nab T_o|^2)ds\le I_{\bbW}$ we get the uniform estimate \[|\nab T_o(t)|^2\le e^{\alpha_1 (I_{\bbW}+1)}\bl(I_{\bbW}+\frac{\alpha_2}{2}\br)=:\rho_{\bbQS}^2\] for all $t\ge t_0+1$. Integrating now inequality \eqref{estTo} in the interval $[t-1,t]$ and using the uniform bounds in time we have found for $|\nab T_o|$, $\Vert\bm\psi\Vert_{\bmH^1}$ and $\Vert\bm\psi\Vert_{\bmH^2}$ one can show that there exists a positive constant $I_{\bbQS}$ such that $\int_t^{t+1}|\Delta T_o(s)|^2ds\le I_{\bbQS}$ for all $t\ge t_0+2$.

Finally it remains to find an uniform bound for $\bm{\psi}$ in $\bmH^2_x$. We multiply in $\bm{L}^2_x$ the equations for $\bm{\psi}$ with $\Delta\bm{\psi}$ as we did when we discussed the existence of the Strong Solution in Section~\ref{Strsec}, where we have seen that all the Jacobian terms cancel. Hence we need to control only the terms $\beta \big(\partial_x \bm{\psi},\Delta\bm{\psi}\big)_{\bm{L}^2}$ and $\big(\omega,\Delta\psi_a^c\big)$, since the other remaining terms have the right sign and can be dropped.
The first is easy, while for the second one we need to estimate the various terms in the equation for $T_a$, using the uniform bounds of $\bm\psi$ and $T_o$ in $H^1$. For the transport term $J(\psi_a^t,T_a)$ we can use the property \eqref{estJ1},
while for the long-wave radiation we can use the embedding of $H^1$ in $L^8$ and relation \ref{eq:Tb}.
Finally, we can estimate $(R_a(T_a),\Delta T_a)$ using Young's inequality and the hypotheses $|R_a(T_a)|^2\le c_1+c_2(|T_a|^2+|\nab T_a|^2)$.

The final estimate will therefore have the following form:
\begin{align}
    \label{fortenuova}
    \pdv{}{t} \biggl(\Vert \bm{\psi}\Vert_{\bmH^2}^2+\kappa\frac{|\nab\psi_o|^2}{L^2_R}&+\mu\gamma_a|\nab T_a|^2\biggr)+\nu_S| A^{3/2}\bm{\psi}|^2\le c_1(\rho_{\bbW},\rho_{\bbQS})\Vert \bm{\psi}\Vert_{\bmH^2}^2+c_2
\end{align} 
where we relabelled $c_1$ and $c_2$ positive constants. Now we can apply the uniform Gronwall's Lemma to get
$\Vert \bm{\psi}(t)\Vert_{\bm{H}^2}^2
\le c_1 I_{\bbW}+\frac{c_2}{2}=:\rho_{\bbS}^2$ for all $t\ge t_1$.
Moreover, integrating inequality \eqref{fortenuova} from $t$ to $t+1$ we easily find a constant $I_{\bbS}$ such that 
$\int_t^{t+1}|A^{3/2} \bm\psi(s)|^2ds\le I_{\bbS}$ for all $t\ge t_1+1.$
Hence we have shown that the ball with maximum radius between $\rho_{\bbQS}$ and $\rho_{\bbS}$ is an absorbing set in $\bbS$. Therefore, since $\bbS$ is compactly embedded into $\bbW$, we proved that this ball is a compact absorbing set  for the semidynamical system generated by the solutions of MAOOAM in $\bbW$.
\subsection{Injectivity of the semigroup on the attractor}
\label{injectsect}
In this subsection, we want to show that on the attractor the dynamics actually define a dynamical system, that is $\restr{S(t)}{\mathcal{A}}$ makes sense for all $t\in \bbR$. This follows from the property of injectivity of $S(t)$ on $\mathcal{A}$, see \cite[Theorem 10.6]{Robinson}, which will be shown in this subsection. However, we cannot directly apply \cite[Lemma 11.9]{Robinson} to obtain injectivity, because in our context the role of $H$ is taken by $\bbW$ and the role of $V$ by $\bbS$, which do not have the structure  $V\subset H \simeq H^* \subset V^*$ required in the aforementioned lemma. Therefore, we will develop a modified version of this lemma.

In the previous subsection we showed that the attractor is a subset of $\bbS$ and hence a weak solution $(\bm{\psi},T_o)$ starting on the the global attractor has the following higher regularity: $\bm{\psi}\in \bmL^{\infty}_t\bmD(A)\cap \bmL^2_t \bmD(A^{3/2})$, $T_o-T_o^{(0)}\in L^{\infty}_t D(A^{1/2})\cap L^2_t D(A)$.

Let us first show, taking prototypical Jacobians terms, some estimates useful in the proofs of Lemma \ref{lemlambda} and Theorem~\ref{teoinjfin}. Applying \eqref{Jpropfond}, H\"older inequalities and the embedding of $H^1$ in $L^4$ and of $H^2$ in $L^{\infty}$ we find
\begin{align}
\label{ineq:k1}
|J(\tilde{\psi}^t,\Delta\tilde{\psi}^c)-J(\hat{\psi}^t,\Delta\hat{\psi}^c)|_{H^{-1}}&
\le|\nab\delta\psi^t|_{L^{4}}|\Delta\tilde{\psi}^c|_{L^4}
+c |\nab\hat{\psi}^t|_{L^{\infty}}|\Delta\delta\psi^c|
\le k_1(t)|\Delta\bm{\delta\psi}|,
\end{align}
where $k_1(t)=c_1(|A^{3/2}\tilde{\psi}^c|+|A^{3/2}\hat{\psi}^t|)$ is square integrable in time since $\hat{\psi}^t,\tilde{\psi}^c\in L^2_tD(A^{3/2})$. 
Instead, the estimates for the nonlinear terms in the equation for $T_o$ will be taken in $L^2_x$. So the Jacobian terms are estimated using directly H\"older inequality as follows
\begin{align}
\label{ineq:k2}
|J(\tilde{\psi}_o,\tilde{T}_o)-J(\hat{\psi}_o,\hat{T}_o)|
&\le |\nab\delta\psi_o|_{L^4}|\nab\tilde{T}_o|_{L^4}+|\nab\hat{\psi}_o|_{L^{\infty}}|\nab\delta T_o|\notag 
\\ &\le 
k_2(t)\bl(\frac{\tilde{\nu}_T}{\gamma_o}|\delta T_o|_{H^1}+\nu_S|\Delta\delta\psi_o|\br),
\end{align} where $k_2(t)= c_2(|\Delta\tilde{T}_o|+|A^{3/2}\hat{\psi}_o|) \in L^2_t$ since $\hat{\psi}_o\in L^2_tD(A^{3/2})$ and $\tilde{T}_o\in L^2_tD(A)$. Finally, for the long-wave radiation we can use the embedding of $H^1$ in $L^4$ and for the short-wave radiation is enough to assume $R_o$ Lipschitz continuous from $H^1$ to $L^2$ with respect to the temperature in order to get
\begin{align}
\label{ineq:k3}
&||\tilde{T}_o|_{\bbR}^3\tilde{T}_o-|\hat{T}_o|_{\bbR}^3\hat{T}_o|+|R_o(\tilde{T}_o)-R_o(\hat{T}_o)|
\le k_3(t)|\delta T_o|_{H^1},
\end{align} where $k_3(t)=c_3(1+|\tilde{T_o}|^3_{H^1}+|\hat{T_o}|^3_{H^1})\in L^2_t$ because $\tilde{T}_o$ and $\hat{T}_o$ are uniformly bounded in time in $H^1$ on the attractor and the same estimate hold with $T_a$ instead of $T_o$.

Before starting with the first lemma, we will fix some notations. Taking the positive operator
\begin{equation}
\label{operD} \mathcal{D}=\bl(-\Delta,-\Delta+\frac{2f_0^2}{p_{\delta}^2\sigma},-\Delta+\frac{1}{L^2_R}\br),
\end{equation}
we can write the equations for the difference between two weak solutions $\bm{\delta\theta}:=(\bm{\delta\psi},\delta T_o)$ in the following compact notation:
\begin{align}
\partial_t \bm{\delta\psi}&=-\nu_S\mathcal{D}^{-1}\Delta^2\bm{\delta\psi}-\mathcal{D}^{-1}f(\bm{\delta\theta})\label{eqabbrpsi}\\
\partial_t \delta T_o&=\frac{\tilde{\nu}_T}{\gamma_o}\Delta \delta T_o +g(\bm{\delta\theta}), \label{eqabbrTo}
\end{align}
where $f$ and $g$ contain all the remaining terms in the equations for $\bm{\delta\psi}$ and $\delta T_o$.
We also define the following norms
\begin{align}
\label{deftildeW}
\Vert \bm{\delta\theta}\Vert_{\tilde{\bbW}}^2:=(\mathcal{D}\bm{\delta\psi},\bm{\delta\psi})+(\delta T_o,\delta T_o) \quad \mbox{and}\quad
\Vert\bm{\delta\theta}\Vert^2_{\nu\bbS}:= \nu_S|\Delta\bm{\delta\psi}|^2+\frac{\tilde{\nu}_T}{\gamma_o}|\nab \delta T_o|^2.
\end{align}
The main result of this section is the following:
\begin{thm}
    \label{teoinjfin}
    Let $\bm{\delta\theta}(t)$ be the difference of two weak solutions $(\bm{\tilde\psi},\tilde{T}_o)$ and $(\bm{\hat\psi},\hat{T}_o)$ of MAOOAM starting on the attractor. Let $R_a(T_a)$ and $R_o(T_o)$ be Lipschitz continuous from $H^1$ to $L^2$. If $\bm{\delta\theta}(t_*)=0$ for some $t_*>0$ then $\bm{\delta\theta}(t)=0$ for all $0\le t\le t_*$, namely the semigroup $S(t)$ is injective on the attractor $\mathcal{A}$. As a consequence, the weak solutions of MAOOAM generate a dynamical system on $\mathcal{A}$.
\end{thm}
\begin{proof}
    We prove the theorem by contradiction: assume that there exists $t_1 \in [0,t_*)$ such that $\bm{\delta\theta}(t_1)\neq 0$. Since
    $\bm{\delta\theta}\in C([0, t_*]; \bmD(A)\times H_0^1)$
    there exists a time $t_2$ for which 
    \[ \Vert\bm{\delta\theta}(t)\Vert_{\tilde{\bbW}}\neq 0 \quad \text{on}\ [t_1,t_2) \qquad \text{and} \quad \bm{\delta\theta}(t_2)=0.\]
    We can take the function $t\mapsto \ln \frac{1}{\Vert\bm{\delta\theta}(t)\Vert_{\tilde{\bbW}}}$. Differentiating it and using relations \eqref{eqabbrpsi} and \eqref{eqabbrTo} we can write
    \begin{align*}
        \dv{}{t} \ln \frac{1}{\Vert\bm{\delta\theta}\Vert_{\tilde{\bbW}}}&\le \frac{\nu_S|\Delta\bm{\delta\psi}|^2+\frac{\tilde{\nu}_T}{\gamma_o}|\nab \delta T_o|^2}{\Vert\bm{\delta\theta}\Vert^2_{\tilde{\bbW}}}+\frac{1}{\Vert\bm{\delta\theta}\Vert^2_{\tilde{\bbW}}}(\langle f(\bm{\delta\theta}),\bm{\delta\psi}\rangle-(g(\bm{\delta\theta}),\delta T_o))\\
        &\le \frac{\Vert\bm{\delta\theta}\Vert^2_{\nu\bbS}} {\Vert \bm{\delta\theta}\Vert_{\tilde{\bbW}}^2} +\frac{c}{\Vert\bm{\delta\theta}\Vert^2_{\tilde{\bbW}}}(\Vert f(\bm{\delta\theta})\Vert_{H^{-1}}+|g(\bm{\delta\theta})|)(|\mathcal{D}^{1/2} \bm{\delta\psi}|+|\delta T_o|).
    \end{align*}
    Therefore setting 
    \begin{equation}
    \label{eq:Lambda}
    \Lambda := \frac{\Vert\bm{\delta\theta}\Vert^2_{\nu\bbS}} {\Vert \bm{\delta\theta}\Vert_{\tilde{\bbW}}^2}
    \end{equation}
    and using estimates for the nonlinear terms like \eqref{ineq:k1},\eqref{ineq:k2} and \eqref{ineq:k3}, it follows that 
    \begin{align*}
       \dv{}{t} \ln \frac{1}{\Vert\bm{\delta\theta}\Vert_{\tilde{\bbW}}} 
        &\le \Lambda+ k(t)\Lambda^{1/2}.
    \end{align*}
    Applying Young's inequality to $k(t)\Lambda^{1/2}$ and integrating between $t_1$ and $t\in [t_1,t_2)$ we can bound the right hand side extending the integral to $t_*$ 
    \begin{align}
    \label{ineq:1/ln}
        \ln \frac{1}{\Vert\bm{\delta\theta}(t)\Vert_{\tilde{\bbW}}}
        &\le \ln \frac{1}{\Vert\bm{\delta\theta}(t_0)\Vert_{\tilde{\bbW}}} +\int_{t_1}^{t_*} (2\Lambda(s) +k^2(s))ds .
    \end{align}
    Now, since $k\in L^2_t$, if we show that $\Lambda$ is integrable in time, inequality \eqref{ineq:1/ln} will give a uniform bound in time for $\frac{1}{\Vert\bm{\delta\theta}(t)\Vert_{\tilde{\bbW}}}$. The property of integrability of $\Lambda$ is an immediate consequence of the upcoming \cref{lemlambda}.

    Therefore we have proved that $\frac{1}{\Vert\bm{\delta\theta}(t)\Vert_{\tilde{\bbW}}}$ is uniformly bounded over $[t_1,t_*)$, but this is in contradiction with our initial assumption that $\bm{\delta\theta}(t_2)=0$.
\end{proof}
\begin{lem}
\label{lemlambda}
    Let $\bm{\delta\theta}(t)$ satisfy the assumptions of \cref{teoinjfin}. Then, given $\Lambda$ as in \eqref{eq:Lambda}, there exists $k(t)\in L^2_t$ such that we have 
    \[ \Lambda(t)\le \Lambda(0)\exp(2\int_0^t k^2(s)ds) \quad \text{ for all }\ t>0.\]
\end{lem}
\begin{proof}
We compute the derivative of $\Lambda$: using the definition of $\Lambda$, \eqref{eqabbrpsi} and \eqref{eqabbrTo} we get
\begin{equation}\label{derlambda}
\begin{split}
\frac{1}{2}\dv{\Lambda}{t}=\frac{1}{\Vert \bm{\delta\theta}\Vert^2_{\tilde{\bbW}}}\br\{ \langle \nu_S\Delta^2\bm{\delta\psi}&-\Lambda \mathcal{D}\bm{\delta\psi},
-\nu_S\mathcal{D}^{-1}\Delta^2\bm{\delta\psi}-\mathcal{D}^{-1}f(\bm{\delta\theta})
\rangle\\ 
\quad &-\bl(\frac{\tilde{\nu}_T}{\gamma_o}\Delta \delta T_o+\Lambda \delta T_o,\frac{\tilde{\nu}_T}{\gamma_o}\Delta \delta T_o +g(\bm{\delta\theta})\br)\br\}.
\end{split}
\end{equation}
Using the definition of $\Lambda$, we observe that 
\begin{align*}
    & 2\Lambda \bl(\nu_S\langle\mathcal{D}^{-1/2}\Delta^2\bm{\delta\psi},\mathcal{D}^{1/2}\bm{\delta\psi}\rangle+\frac{\tilde{\nu}_T}{\gamma_o}|\nab \delta T_o|^2\br)-\Lambda^2 (|\mathcal{D}^{1/2}\bm{\delta\psi}|^2+|\delta T_o|^2)=\\
    &=2\Lambda \bl(\nu_S|\Delta\bm{\delta\psi}|^2+\frac{\tilde{\nu}_T}{\gamma_o}|\nab \delta T_o|^2\br)-\Lambda^2 (|\mathcal{D}^{1/2}\bm{\delta\psi}|^2+|\delta T_o|^2)=\Lambda\bl(\nu_S|\Delta\bm{\delta\psi}|^2+\frac{\tilde{\nu}_T}{\gamma_o}|\nab \delta T_o|^2\br).
\end{align*} 
Hence we can rewrite the parts which do not contain $f$ and $g$ as
follows:
\begin{equation}
    \label{derivlambda2}
    \begin{split}
    \dv{\Lambda}{t}=&-\frac{|\nu_S\mathcal{D}^{-1/2}\Delta^2\bm{\delta\psi}-\Lambda \mathcal{D}^{1/2}\bm{\delta\psi}|^2}{\Vert \bm{\delta\theta}\Vert^2_{\tilde{\bbW}}}-\frac{|\frac{\tilde{\nu}_T}{\gamma_o}\Delta \delta T_o +\Lambda \delta T_o|^2}{\Vert \bm{\delta\theta}\Vert^2_{\tilde{\bbW}}}\\
    &+\frac{1}{\Vert \bm{\delta\theta}\Vert^2_{\tilde{\bbW}}}\langle \nu_S\mathcal{D}^{-1/2}\Delta^2\bm{\delta\psi}-\Lambda \mathcal{D}^{1/2}\bm{\delta\psi}, \mathcal{D}^{-1/2}f(\bm{\delta\theta})\rangle\\
    & +\frac{1}{\Vert \bm{\delta\theta}\Vert^2_{\tilde{\bbW}}}\bl(\frac{\tilde{\nu}_T}{\gamma_o}\Delta \delta T_o+\Lambda \delta T_o,g(\bm{\delta\theta})\br).
    \end{split}
\end{equation}
In order to estimate the term in the second line of \eqref{derivlambda2} we use Young's inequality to counterbalance the first part by the dissipative term, and the second part is estimated by $\Vert f(\bm{\delta\theta})\Vert^2_{H^{-1}}\le k_f^2(t)\Vert \bm{\delta\theta}\Vert^2_{\nu\bbS}$ which follow from inequalities like \eqref{ineq:k1}. The third line of \eqref{derivlambda2} is treated analogously using inequalities \eqref{ineq:k2} and \eqref{ineq:k3}. In conclusion, by definition of $\Lambda$, we get
\begin{align*}
    &\dv{\Lambda}{t}+\frac{1}{\Vert \bm{\delta\theta}\Vert^2_{\tilde{\bbW}}}\bl(|\nu_S\mathcal{D}^{-1/2}\Delta^2\bm{\delta\psi}-\Lambda \mathcal{D}^{1/2}\bm{\delta\psi}|^2+|\frac{\tilde{\nu}_T}{\gamma_o}\Delta \delta T_o +\Lambda \delta T_o|^2\br)\le 2k^2(t)\Lambda
\end{align*}
where $k^2(t)$ is integrable in time, and the result follows immediately from Gronwall's Lemma.
\end{proof}

\subsection{Dimension of the Attractor}
\label{dimattr}
In order to estimate the dimension of the attractor we have to study the evolution of infinitesimal $n$-dimensional volumes of the phase space $\bbW$ as they evolve under the flow and try to find the smallest number $n$ at which we can guarantee that all such $n$-volumes contract asymptotically in time with an exponential rate. 

We will follow the ideas presented in \cite[Chapter 6]{Temam}, and Robinson \cite[Chapter 13]{Robinson}: we start showing that the flow generated by our equations is uniformly differentiable. To do this
we first have to study the linearized dynamics of MAOOAM. 

We will denote with $(\bm{\tilde{\psi}},\tilde{T}_o)$ a generic weak solution of MAOOAM, and with $(\bm{\Psi}(t),\rom{T}_o(t))$ the solution of the linearized version of MAOOAM around $(\bm{\tilde{\psi}},\tilde{T}_o)$. We denote by $\tilde{\Lambda}(t, \bm{\tilde{\psi}}(0),\tilde{T}_o(0))$  the operator representing the evolution of $(\bm{\Psi}(0),\rom{T}_o(0))$ around$(\bm{\tilde{\psi}},\tilde{T}_o)$.

One can prove using the same techniques we have used to find the Weak and Strong solutions for MAOOAM in Sections~\ref{Chapweak} and \ref{Strsec} that $(\bm{\Psi}(t),\rom{T}_o(t))$ exists, is unique, and depends Lipschitz-continuously on the initial datum 
in the weak and strong spaces $\bbW$ and $\bbS$, defined in \eqref{defWeS}. 
\begin{prop}
    \label{undif}
    Assume that $|\partial_{T_i} R_i(T_i)|_{L^4}+|\partial^2_{T_i} R_i(T_i)|_{L^4}\le k_1+k_2|T_i|$, where $i=a,o$ and $k_1,k_2$ are positive constants. We denote the partial derivative with respect to the temperature of $R(T)$ with $\partial_T R$. Then $\tilde\Lambda(t; \bm{\tilde{\psi}}(0),\tilde{T}_o(0))$ is compact as an operator on $\bbW$ for all $t\ge0$ and the flow generated by the equations for MAOOAM is uniformly differentiable on $\mathcal{A}$ in $\bbW$.
\end{prop}
\begin{proof}
The compactness of $\tilde\Lambda(t; \bm{\tilde{\psi}}(0),\tilde{T}_o(0))$ can be shown using the strategy we followed when we have proved the compactness of the absorbing set in Section~\ref{existattr}.

Recall that  $(\bm{\tilde{\psi}},\tilde{T}_o)$ is the fixed weak solution around which we linearize and $(\bm\Psi,\rom{T}_o)$ is the solution of the linearized equations. Denote by  $(\bm{\hat{\psi}},\hat{T}_o)$
another weak solution and by $(\bm{\delta\psi},\delta T_o)=(\bm{\tilde{\psi}}-\bm{\hat{\psi}},\tilde{T}_o-\hat{T}_o)$ the difference between the two weak solutions. We can now define 
\begin{equation}
    \label{eq:lin:var1}
    \bm{\Phi}=\bm{\delta\psi}-\bm{\Psi}, \quad \mathcal{T}_o=\delta T_o-\rom{T}_o.
\end{equation}
All these fulfill Dirichlet boundary condition. 
As usual we take the scalar product of the equations for $\bm{\Phi}$ in $\bmL^2$ with $\bm{\Phi}$ and of the equation for $\mathcal{T}_o$ in $L^2$ with $\mathcal{T}_o$.
From now on we will use the reduced norm $\Vert \bm{\Phi}\Vert_{\bmH^1}^2+|\mathcal{T}_o|^2$ instead of $\Vert (\bm\Phi, \mathcal{T}_o)\Vert^2_{\bmH_0^1\times L^2}$.
In order to prove that the flow is differentiable on $\mathcal{A}$ it is sufficient to show that 
\begin{align}
\label{undiflim}
\frac{\Vert\bm{\Phi}(t)\Vert_{\bm{H}^1}^2+|\mathcal{T}_o(t)|^2}{\Vert\bm{\delta\psi}(0)\Vert_{\bm{H}^1}^2+|\delta T_o(0)|^{2}}\le \Vert\bm{\delta\psi}(0)\Vert_{\bm{H}^1}^{2}+|\delta T_o(0)|^{2}.
\end{align}
To get the above relation it will be enough to prove the following inequality:
\begin{align}
\label{gpqz}
\pdv{}{t} \bigl(\Vert\bm{\Phi}(t)\Vert_{\bm{H}^1}^2+|\mathcal{T}_o(t)|^2\bigr)\le C_1&(\Vert\bm{\Phi}(t)\Vert_{\bm{H}^1}^2+| \mathcal{T}_o(t)|^2)\notag\\
&+C_2(\Vert\bm{\delta\psi}(t)\Vert_{\bm{H}^1}^2\Vert\bm{\delta\psi}(t)\Vert_{\bm{H}^2}^2+|\delta T_o(t)|^2|\nab\delta T_o(t)|^2),
\end{align}
where $C_1$ and $C_2$ are constants that depend only on the radii of the absorbing sets. Indeed, the integral of the terms in the second line of  \eqref{gpqz} is bounded by $\Vert\bm{\delta\psi}(0)\Vert_{\bm{H}^1}^4 +|\delta T_o(0)|^4$ as a consequence of the estimates used to show uniqueness of the weak solution in Section~\ref{uniq}. An application of Gronwall's Lemma  gives us then \eqref{undiflim}.

Let us show now  \eqref{gpqz}. We will adopt the following strategy: we isolate the terms with $\bm{\delta\psi}$ and $\delta T_o$ and we uniformly bound the ones with $\bm{\tilde{\psi}},\bm{\hat{\psi}}$, $\tilde{T}_o$ and $\hat{T}_o$ using the radii of the absorbing sets we found in Section~\ref{existattr}. 
In the equations for $\bm{\Phi}$ the main issue is to find a bound for the terms involving the Jacobian. Using the inverse of relation \eqref{Jcambcord1} we can limit ourself to Jacobians with the streamfunctions of the same type in both its arguments; the types being oceanic, baroclinic, barotropic, first and third of the atmosphere streamfunctions. 
For each of these types we can rewrite all the associated Jacobian terms as follows
\begin{align}
\label{unidifJ}
    J(\tilde{\psi},\Delta \tilde{\psi})-J(\hat{\psi},\Delta\hat{\psi})-J(\tilde{\psi},\Delta\Psi)-J(\Psi, \Delta\tilde{\psi})=J(\Phi, \Delta\tilde{\psi})+J(\tilde{\psi}, \Delta \Phi)-J(\delta\psi,\Delta\delta\psi).
\end{align}
Multiplying the above equation with $\Phi$ using inequality \eqref{estJ1} we find
\begin{align}
    &(J(\Phi,\Delta \tilde{\psi}),\Phi)=0,\notag\\
    &(J(\tilde{\psi}, \Delta \Phi),\Phi)\le c_1|\Delta\tilde{\psi}|^2|\nab \Phi|^2+ \frac{\nu_S}{16}|\Delta\Phi|^2 \leq c_2\rho_{\bbS}^2|\nab \Phi|^2 +\frac{\nu_S}{16}|\Delta\Phi|^2,\notag\\
    &(J(\delta\psi,\Delta\delta\psi),\Phi)\le c_2|\nab\delta\psi|^2|\Delta\delta\psi|^2+\frac{\nu_S}{16}|\Delta\Phi|^2.\label{eq:J3}
\end{align}
In the estimate \eqref{eq:J3}, it is essential to have the same type of $\delta\psi$ to get the form required in \eqref{gpqz}.

An analogous relation to \eqref{unidifJ} holds also for the Jacobian terms related to the temperature functions, and then similar estimates to the ones written above will hold. %

Let us now consider the infrared radiation terms: we will concentrate on the outgoing radiation for the ocean $|\tilde{T}_o|^3_{\bbR}\tilde{T}_o-|\hat{T}_o|^3_{\bbR}\hat{T}_o-4|\tilde{T}_o|^3_{\bbR}\rom{T}_o$ as it is the hardest to estimate. This quantity can be rewritten as $6\xi^2\delta T_o^2+4|\tilde{T}_o|^3\mathcal{T}_o$ applying Taylor's expansion with Lagrange's remainder, where $\xi:= \theta \tilde{T}_o +(1-\theta)\hat{T}_o$ and $\theta$ is a real function with values in $[0,1]$.
Multiplying with $\mathcal{T}_o$ we get two summands that can be bounded as follows:
\begin{align*}
    (\xi^2\delta T_o^2,\mathcal{T}_o)&\le c_4|\xi^2|_{L^4}^2|\delta T_o|^2|\nab\delta T_o|^2+\frac{\nu_T}{16}|\nab \mathcal{T}_o|^2,\\
    (|\tilde{T}_o|_{\bbR}^3\mathcal{T}_o,\mathcal{T}_o)&\le c_5|\tilde{T}_o|^6_{H^1}|\mathcal{T}_o|^2+\frac{\nu_T}{16}|\nab\mathcal{T}_o|^2,
\end{align*}
where the prefactors $|\xi^2|_{L^4}^2$ and $|\tilde{T}_o|^6_{H^1}$ can be uniformly estimated by the radii of the absorbing sets. The other three infrared radiation functions can be estimated in the same way.

Finally, it remains to estimate the short-wave radiation: using again Taylor's formula with Lagrange's remainder and the assumptions on $R_a$ and $R_o$
we get
\begin{align*}
    (R_o(\tilde{T}_o)-&R_o(\hat{T}_o)-\partial_T R_o(\tilde{T}_o)\rom{T}_o,\mathcal{T}_o)\\
    &\le c_7\big((2k_1+k_2|\xi|^2+ k_2|\tilde{T}_o|^2)|\mathcal{T}_o|^2+|\delta T_o|^2|\nab\delta T_o|^2\big)+\frac{\nu_T}{16}|\nab \mathcal{T}_o|^2,
\end{align*}
and the same holds for $R_a$. Putting all these estimates together we get inequality \eqref{gpqz}.
\end{proof}

Now we want to show that the  attractor $\mathcal{A}$ has finite fractal dimension.
Taking the positive self-adjoint operator $\mathcal{D}$ we defined in \eqref{operD}, 
we can write the linearized equations for MAOOAM in the following compact form: 
\begin{align}
    \label{teta}
   \pdv{}{t}\binom{\bm{\Psi}}{\rom{T}_o} =
     L\left(t;\bm{\tilde{\psi}}(0),\tilde{T}_o(0)\right)   
\end{align}
where 
\begin{equation}\label{eq:linearized:L}
    L\left(t;\bm{\tilde{\psi}}(0),\tilde{T}_o(0)\right) := \left(\begin{array}{c}
    -\nu_S\mathcal{D}^{-1}\Delta^2\bm{\Psi}-\mathcal{D}^{-1}L_{\Psi}(t;\bm{\tilde{\psi}}(0))\bm{\Psi},
     \\ \frac{\tilde{\nu}_T}{\gamma_o}\Delta\rom{T}_o+L_{\rom{T}_o}(t;\tilde{T}_o(0))\rom{T}_o
     \end{array} \right), 
\end{equation}
and the term $L_{\Psi}(t;\bm{\tilde{\psi}}(0))$ and $L_{\rom{T}_o}(t;\tilde{T}_o(0))$ summarise all the terms other than dissipation in the linearized equations for the streamfunctions and for the ocean temperature respectively.

In order to show that the attractor does not contain $n$-dimensional subsets, and then its dimension is smaller than $n$, the infinitesimal volumes $V_n(t)=|\bm{\delta x}^{(1)}\wedge\dots\wedge\bm{\delta x}^{(n)}|$ formed by the infinitesimal displacements $\{\bm{\delta x}^{(i)}(t)\}_{i=1,\dots,n}$ around $(\bm{\Psi},\rom{T}_o)$ must contract with negative exponential rate independently of the initial point $\bm{x}_0:=(\bm{\tilde{\psi}}(0),\tilde{T}_o(0))$ and of the initial infinitesimal $n$-volume, which is specified by $\Pi_n(0)$. 

These properties can be expressed in term of the following quantity. Denote $\langle g\rangle$ to be the asymptotic time average $\limsup_{t\rightarrow \infty}\frac{1}{t}\int_0^t g(s)ds$, $\text{Tr}$ the trace in $\bbW$ and $\Pi_n(t)$ the orthogonal projection in $\bbW$ on the space spanned by $\{\bm{\delta x}^{(i)}(t)\}_{i=1,\dots,n}$. Then we define
\begin{equation}\label{eq:def::calTr}
    \mathcal{TR}_n(\mathcal{A}):=\sup_{\bm{x}_0\in \mathcal{A}}\sup_{\Pi_n(0)}\langle\text{Tr}(L(\bm{x}_0)\Pi_n )\rangle.
\end{equation}
The sign of such a quantity heuristically represents the decay or growth rate of the infinitesimal volume $V_n(t)$. We can now show the following:
\begin{lem}\label{lemma:tr<0}
Let $\bm{\Phi}$ and $\mathcal{T}_o$ be as in \eqref{eq:lin:var1} and $L$ as in \eqref{eq:linearized:L}. Then there exists $n_*$ such that $ \mathcal{TR}_n(\mathcal{A}) < 0$ for all $n >n_*$.
\end{lem}
\begin{proof}
We take a set of functions $\{(\bm{\phi}_i,\tau_{o,i})\}_i$ in $\bbW$ orthonormal with respect to the scalar product defined by the operator $\mathcal{D}$, which we have previously denoted with $\tilde{\bbW}$, see \eqref{deftildeW}.
If we define 
\begin{align*}
\rho(x,y):=\sum_{j=1}^n (|\nab\phi^t_j(x,y)|^2_{\bbR^2}&+|\nab\phi^c_j(x,y)|^2_{\bbR^2}+\frac{2f_0^2}{p_{\delta}^2\sigma}|\phi^c_j(x,y)|^2_{\bbR}\\&+|\nab\phi_{o,j}(x,y)|^2_{\bbR^2}+\frac{1}{L^2_R}|\phi_{o,j}(x,y)|^2_{\bbR}+|\tau_{o,j}(x,y)|_{\bbR}^2),
\end{align*}
using the orthonormality of $\bm\phi_j$ and $\tau_j$ we immediately get
$\int_{\Omega} \rho(x,y)\ d\tilde{A}=n$.
We start computing the trace of the diffusive terms. Defining $\nu=\min\{\nu_S,\frac{\tilde{\nu}_T}{\gamma_o}\}$  
and using the Sobolev-Lieb-Thirring inequality for $\rho$ (see Temam \cite[Appendix 4.2]{Temam})  we get for some constants $\tilde{k}_1$ and $\tilde{k}_2$
\[ 
    \sum_{j=1}^n \bl\langle\mathcal{D}^{-1}(-\nu_S\Delta^2 \bm{\phi}_j, \frac{\tilde{\nu}_T}{\gamma_o}\Delta\tau_{o,j}),(\bm{\phi},\tau_o)_j\br\rangle_{\tilde{\bbW}}\le -\nu\frac{n^2}{\tilde{k}_2}+ \nu\frac{\tilde{k}_1}{\tilde{k}_2}n.
\]
Now we want to show that the trace of the other terms are of order $n$, as then for $n$ large enough $\langle\text{Tr}(L(\bm{x}_0)\Pi_n\rangle<0$. 
As a prototype for the Jacobian terms let us consider  the following:
\begin{align*}
    &\sum_{j=1}^n \bl\langle \bl(\frac{2f_0^2}{p_{\delta}^2\sigma}-\Delta\br)^{-1}J(\phi_j^{t},\Delta\psi^c),\phi_j^{c}\br\rangle_{\tilde{\bbW}}=\sum_{j=1}^n( J(\phi_j^{t},\Delta\psi^c),\phi_j^{c})\\
    &\le \sum_{j=1}^n \left( |\nab\phi_j^{t}||\Delta\phi_j^{t}||\nab\phi_j^{c}||\Delta\phi_j^{c}|\right)^{1/2} |\Delta\psi^c|
    \le C_1\rho_{\bbS}^2n+\frac{\nu}{8}\sum_{j=1}^n(|\Delta\phi_j^{t}|^2+|\Delta\phi_j^{c}|^2),
\end{align*}
where we used H\"older, Ladyzhenskaya, uniform bounds on the attractor and $\sum_{j=1}^n \Vert\bm{\phi}_j\Vert_{\bm{H}^1}^2\le n$.
All the other Jacobian terms can be treated with similar techniques.

For the radiation terms the estimates are even simpler. We can take as prototypes the ones for $T_o$: the infrared radiation can be estimated as follows
\begin{align*}
\sum_{j=1}^n(|T_o|^3_{\bbR}\tau_{o,j},\tau_{o,j})\le |T_o|^3_{L^6}\sum_{j=1}^n|\tau_{o,j}|^2_{L^4}
\le C_3\rho_{\bbQS}^6n+\frac{\nu}{8}\sum_{j=1}^n|\nab\tau_{o,j}|^2.
\end{align*}
Assuming $|\partial_T R_o(T_o)|\le k_1+k_2|T_o|$ we can control the short-wave radiation by
\begin{align*}
    &\sum_{j=1}^n (\partial_T R_o(T_o)\tau_{o,j}, \tau_{o,j})\le C_4|\partial_T R_o(T_o)|^2\sum_{j=1}^n|\tau_{o,j}|^2_{L^4},
\end{align*}
which can be bounded as above. 
Collecting all the estimates together we get \begin{align*}
\text{Tr}(L(t;\bm{\psi},T_o)\Pi_n(t))\le \sum_{j=1}^n \langle\mathcal{D}^{-1}(-\nu_S\Delta^2 \bm{\phi}_j, \frac{\tilde{\nu}_T}{\gamma_o}\Delta\tau_{o,j}),(\bm{\phi},\tau_o)_j\rangle_{\tilde{\bbW}}+ c_0n\le -c_1n^2+c_2n,
\end{align*}
then choosing $n>c_2/c_1=: n_*$ one gets $\langle\text{Tr}(L(\bm{x}_0)\Pi_n)\rangle<0$ for all initial points in the attractor and as required.
\end{proof}

Summarizing, thanks to \cref{lemma:tr<0} together with \cref{undif} the fractal dimension of the attractor is finite since $\mathcal{A}$ cannot contain $n$-dimensional subsets. This implication is made rigorous for example in \cite[Theorem 13.16]{Robinson} and for our model then we showed the following:
\begin{thm}
    Let $R_a$ and $R_o$ be such that 
    \begin{equation}
        |\partial_{T_i} R_i(T_i)|_{L^4}+|\partial^2_{T_i} R_i(T_i)|_{L^4}\le k_1+k_2|T_i|, \quad i=a,o
    \end{equation}
     for some $k_1,k_2$ positive constants. Then the global attractor for MAOOAM has finite fractal dimension.
\end{thm}

\section{Determining Modes}
\label{detmod}
In this final section we want to understand whether there exists a finite number of degrees of freedom, frequently called in literature \textit{determining modes}, for the semidynamical system $((\bmH_0^1 \times L^2)(\Omega),\ \{S(t)\}_{t\ge0})$, which determine in a unique way the asymptotic behaviour of the system if we consider trajectories on the attractor.
 Using the framework given in \cite{Chueshov98} we show that there exists a finite number of asymptotically determining functionals for the MAOOAM system. Moreover, we can also prove that we can choose these functionals in such a way that they do not depend on the ocean temperature. 
To show the latter, we have to modify slightly the framework described in \cite{Chueshov98} in order to cover the case of a set of determining functionals that do not depend on one (or more) component of the solution (in our case $T_o$). 

Following \cite{Chueshov98} we do not provide explicitly the asymptotic determining functionals for our model, but we give a sufficient condition for a family to be asymptotically determining for our semidynamical system. More precisely, we consider a family $\mathfrak{L}=\{l_j\ :\ j=1,\dots N\}$, where $l_j$ are linear functionals on $\bbW$ which have the property that   
\begin{multline}
    \label{detfunPoinc1}\Vert\bm{\delta\psi}\Vert_{\bm{H}^1}^2+\kappa\frac{|\delta\psi_o|^2}{L^2_R}+\mu\gamma_a|\delta T_a|^2\le \varepsilon^2_{\mathfrak{L}}\bl(\frac{\nu_S}{2}\Vert\bm{\delta\psi}\Vert_{\bm{H}^2}^2+\frac{r\kappa}{2}|\nab\delta\psi_o|^2+\frac{k'_d R^2}{4f_0^2}|\nab\delta T_a|^2\br)\\+C_{\mathfrak{L}}\max_{j=1,\dots, N}|l_j(\bm{\delta\psi})|^2,
\end{multline}
for two constants $C_{\mathfrak{L}}$ and $\varepsilon_{\mathfrak{L}}$. We will show in \cref{teodetfunct} that when $\varepsilon_{\mathfrak{L}}$ is small enough, i.e. \eqref{small} below, then such $\mathfrak{L}$ are a family of asymptotically determining functionals for our semidynamical system. The condition \eqref{detfunPoinc1} differs from those used in \cite{Chueshov98} as it does not depend on one of the variables of the system, namely the ocean temperature does not appear.

 Typical examples of such functionals are projections on eigenfunctions of the Laplacian, local volume averages or functions evaluated in specific points, see \cite{Chueshov98} for more details. 
In our case \eqref{detfunPoinc1} can be easily shown for example for $\mathfrak{L}$ the family of projections on the first $N$ eigenfunctions of the Laplacian. Moreover when choosing $N$ large enough $\varepsilon_\mathfrak{L}$ becomes small enough to satisfy \eqref{small} and hence for those $N$ we obtain a family of asymptotically determining functionals.

\begin{thm} 
    \label{teodetfunct}
    Assume that the hypotheses ensuring the existence and uniqueness of the Strong Solution are fulfilled and assume that $R_o$ satisfies the bound 
    \begin{align} 
        \label{stimaLlambda}
        (R_o(\tilde{T}_o)-R_o(\hat{T}_o),\delta T_o)\le \mathcal{L}_{R_o}|\delta T_o|^2,
    \end{align}
    with the constant $\mathcal{L}_{R_o}$ such that $\mathcal{L}_{R_o} < \lambda $. Then there exists $\varepsilon^*$ such that if \eqref{detfunPoinc1} holds for some $\varepsilon_{\mathfrak{L}}\in (0, \varepsilon^*)$, then $\mathfrak{L}$ is a set of asymptotically
    determining functionals for MAOOAM, which can be taken independent of the ocean temperature. More precisely if
    \[ 
        \lim_{t\rightarrow\infty} \int_t^{t+1}|l_j(\bm{\delta\psi}(\tau))|^2d\tau=0, \quad j=1,\dots N
    \] 
    then 
    \[ \lim_{t\rightarrow\infty}\Vert(\bm{\delta\psi},\delta T_o)(t)\Vert^2_{\bmH_0^1\times L^2}=0.\]
\end{thm}
Again note that \eqref{stimaLlambda} is in particular fulfilled if $R_o$ is Lipschitz with constant $\mathcal{L}_{R_o}$ but the requirement \eqref{stimaLlambda} allows for more general functions in principle. 

\begin{proof}
We will work again with the equations for the differences $\bm{\delta\psi}$ and $\delta T_o$ of two weak solutions we derived in Section~\ref{uniq}.
As we are deriving an inequality which holds for all $t$ after a fixed time $t_0$, we can work on the absorbing sets, and the idea we follow is that we want to estimate the various terms of the equation in a way to obtain small coefficients in front of $\delta T_o$ and big coefficients in front of $\bm{\delta\psi}$.
We need to avoid big constants in front of $|\delta T_o|^2$, since we cannot balance them using \eqref{detfunPoinc1}. However, for small enough $\varepsilon_{\mathfrak{L}}$ we can counterbalance all big constants in front of $\bm{\delta\psi}$.
As  we do not care about the exact value of the constants in front of terms with $\Vert\bm{\delta\psi}\Vert_{\bm{H}^1}^2$, we will estimate the nonlinear transport terms in the equations for $\bm{\delta\psi}$ with 
\[ C(\rho)\Vert\bm{\delta\psi}\Vert_{\bm{H}^1}^2+\frac{\nu_S}{2}\Vert\bm{\delta\psi}\Vert_{\bm{H}^2}^2,\]
where $C(\rho)$ is a  (potentially very big) constant that depends on $\rho=\max\{\rho_{\bbW},\rho_{\bbS}\}$.

Let us focus now on the terms which depend on the temperature. 
We start with the transport term $(J(\delta\psi_o,\tilde{T}_o),\delta T_o)$: using H\"older, Ladyzhenskaya and twice Young's inequalities we find
\begin{align*}
(J(\delta\psi_o,\tilde{T}_o),\delta T_o)
&\le c|\nab \tilde{T}_o||\tilde{T}_o||\nab\delta\psi_o||\Delta\delta \psi_o|+\frac{\nu_T}{2}|\nab \delta T_o|^2\\
&\le c\rho_{\bbW}^2\rho_{\bbQS}^2|\nab \delta\psi_o|^2+\frac{\kappa\nu_S}{4}|\Delta\delta \psi_o|^2+\frac{\nu_T}{2}|\nab \delta T_o|^2.
\end{align*} 
The term related to heat exchange on the ocean-atmosphere interface gives us a damping term for any choice of  $\tilde{\varepsilon}\in (0,1)$, using
\[ \lambda|\delta T_a-\delta T_o|^2\ge \lambda(1-\tilde{\varepsilon})|\delta T_o|^2-\frac{\lambda(1-\tilde{\varepsilon})}{\tilde{\varepsilon}}|\delta T_a|^2.\]
Next, let us consider  the terms related to the infrared radiation:
the ones describing the outgoing radiation give rise to positive quantities, which we can neglect, while the remaining  two can be controlled using that $\tilde{T}_o,\hat{T}_o$ belong to the absorbing set of radius $\rho_{\bbQS}$, in the following way:
\begin{align*}
\epsilon_a\sigma_B(|\tilde{T}_o|^3_{\bbR}\tilde{T}_o-|\hat{T}_o|^3_{\bbR}\hat{T}_o,\delta T_a)
&\le c(|\tilde{T}_o|^3_{H^1}+|\hat{T}_o|^3_{H^1})|\nab \delta T_a||\nab\delta T_o|\\
&\le c\rho_{\bbQS}^6|\nab\delta\psi^c|^2+\frac{\nu_T}{8}|\nab\delta T_o|^2,
\end{align*}
and a similar estimate holds for the other one.

Finally, thanks to these derived estimates, inequality \eqref{stimaLlambda} for $R_o(T_o)$ and \eqref{detfunPoinc1} we find
\begin{align}  \label{rrrrrrrrr}
\begin{split}
    \frac{1}{2}\pdv{}{t}\biggl(\Vert\bm{\delta\psi}\Vert_{\bm{H}^1}^2&+\kappa\frac{|\delta\psi_o|^2}{L^2_R}+\mu(\gamma_a|\delta T_a|^2+\gamma_o|\delta T_o|^2)\biggr)\\
    +\bigl(\varepsilon_{\mathfrak{L}}^{-2}&+\varsigma-C(\rho)\bigr)\bl(\Vert\bm{\delta\psi}\Vert_{\bm{H}^1}^2+\kappa\frac{|\delta\psi_o|^2}{L^2_R}+\mu\gamma_a|\delta T_a|^2\br)\\
    &+\mu\left(\lambda(1- \tilde{\varepsilon})-\mathcal{L}_{R_o}\right)|\delta T_o|^2\le C_{\mathfrak{L}}\max_{j=1,\dots, N}|l_j(\bm{\delta\psi})|^2,
    \end{split}
\end{align}
where 
\[
    \varsigma:=\min\bl\{\frac{k_d}{2},k'_d, \frac{k'_d\lambda_1R^2}{4f_0^2\mu\gamma_a},\frac{r}{2},\frac{rL^2_R\lambda_1}{2}\br\}
\]
comes from the frictions terms.
We chose $\tilde{\varepsilon}$ so that $\mathcal{L}_{R_o} \leq \lambda(1- 2\tilde{\varepsilon})$, then the prefactor of $|\delta T_o|^2$ is non-negative as it is bounded below by $\mu \lambda\tilde{\varepsilon}$. Last, we see that the condition 
\begin{align}
    \label{small}
    \varepsilon_{\mathcal{L}}^{-2}+\varsigma-C(\rho)>0
\end{align}
is fulfilled for any $\varepsilon_{\mathfrak{L}}$ if $\varsigma-C(\rho)>0$, otherwise we have to choose $\varepsilon_{\mathfrak{L}}$ small enough, namely
\begin{align*}
    \varepsilon_{\mathcal{L}} < \left( C(\rho) - \varsigma\right)^{-1/2}=\varepsilon^*
\end{align*}
Then integrating \eqref{rrrrrrrrr} we derive that the norm contracts with rate 
$    \min\left\{\varepsilon_{\mathfrak{L}}^{-2}+\varsigma-C(\rho), \frac{ \lambda\tilde{\varepsilon}}{\gamma_o }\right\} 
$.
%

\end{proof}
\begin{rem}
\label{remLminlambda}
The assumption $\mathcal{L}_{R_o}< \lambda$ is physically realistic. Indeed, if we take a radiation function, as in \cite{MAOOAM}, which is independent of temperature $T$, the assumption is trivially true because the constant $\mathcal{L}_{R_o}$ is zero.

In the EBM models non-constant radiation functions of the following type are considered: $R(x,y,T)=q(x,y)\beta(T)$, where $q(x,y)$ is the insolation function, e.g. as in \cite[Section 6.3]{Peix},
\[q(x,y)=S\bl(\frac{d_m}{d}\br)^2\cos Z,\]
and $\beta(T)$ is the effective coalbedo of the ocean, e.g.  taken in \cite{KapEng}, \cite{Peix} as a piecewice linear function of the temperature,
\begin{align}
\label{coalb}
\beta(T)=\begin{cases}
\beta_{-}\ &T\le T_{-}\\
\beta_{-}+\frac{T-T_{-}}{T_{+}-T_{-}}(\beta_{+}-\beta_{-})\ &T\in [T_{-},T_{+}]\\
\beta_{+}\ &T\ge T_{+}
\end{cases}
\end{align}
with $\beta_{-}=0.3,\ \beta_{+}=0.7,\ T_{-}=250\ K$ and $T_{+}=280\ K$, see \cite[Chapter 2]{KapEng}. 
Now, $S=1360\ Wm^{-2}$ is the solar constant, $d_m$ and $d$ are respectively the mean and actual distance between Sun and Earth and $Z$ is the Sun's Zenith angle.
Clearly $\cos Z\le 1$ and using the values presented in \cite{Peix} we have  $d_m^2/d^2\in (0.96, 1.04)$.
This gives us an explicit value for the constant $\mathcal{L}_{R_o}$ of the radiation function for the ocean, and using the numerical values mentioned we find
\[ \mathcal{L}_{R_o}\leq 1.04 \times S\frac{\beta_{+}-\beta_{-}}{T_{+}-T_{-}}= 1.36\times 10^3\cdot 1.33\times 10^{-2}=18.1\ WK^{-1}m^{-2}.\]

We conclude that for the value of $\lambda$ used for example in \cite{VannitsemGhil}, namely $\lambda\approx 20\ WK^{-1}m^{-2}$, the relation $\mathcal{L}_{R_o}< \lambda$ holds.
\end{rem}

Theorem~\ref{teodetfunct} shows that larger values of $\lambda$ lead to an enslaving effect on the ocean temperature which is in line with a more general stabilization of the system observed in \cite{VannitsemGhil}, see Figure~14 there. Leaving aside the different boundary conditions among our frameworks, the values of $\lambda$ we consider include those for which \cite{VannitsemGhil} conjecture the system to be chaotic. Moreover, Vannitsem at al. argue that the low-frequency variability is due to the coupling between the ocean and atmosphere temperature parametrised by $d$, the wind drag coefficient. Theorem~\ref{teodetfunct} shows that asymptotically the ocean temperature is determined by the streamfunctions: however, this does not exclude that the ocean temperature, though being a function of the streamfunctions, creates a non-trivial additional coupling between atmosphere and ocean. Indeed, this additional coupling may give rise to the aforementioned low-frequency variability studied in \cite{VannitsemGhil}. 
We can try to find a quantitative estimate for the order of magnitude which appear in our condition for  $\varepsilon_{\mathfrak{L}}$. Let us consider the case of a set of functionals given by the projections on the $N$-dimensional space spanned by the eigenfunctions $e_j$ of $-\Delta$. 
In this case 
estimate \eqref{detfunPoinc1} holds with $\varepsilon^2_{\mathfrak{L}}
=2(\nu_S\lambda_{N+1})^{-1}$, where $\lambda_{N+1}$ is the $(N+1)$-eigenvalue of $-\Delta$ taking the eigenvalues in increasing order. We have 
$\varepsilon_{\mathfrak{L}}\sim L/\sqrt{\nu_S N}$. 
In order to fulfil estimate \eqref{small} (ignoring $\varsigma $ as it is small) we get
\begin{align}
    \label{epsN}
    N\sim\frac{L^2}{\nu_S \varepsilon^2_{\mathfrak{L}}}\ge 10^{12}C(\rho),
\end{align}
where $C(\rho)$ is at least of the same order of $\rho^2_{\bbW}$, then approximately $10^{34}$ using the values we found in Section~\ref{existattr}. This estimate for $N$ appears independently whether we consider functionals which depend on the ocean temperature or not.
We see that, as usual for this technique, the $N$ has to be seen as a worst case scenario overestimate. 
 \begin{rem}
    Proceeding similarly to what we have done to show \cref{teodetfunct}, one can derive conditions for determining functionals that are independent of one or more streamfunctions. In this case the conditions we obtain require the relevant friction coefficients to be large enough. The suitable values of the friction parameters are not physically realistic in the sense that they are  of the same order of magnitude as our bounds on the radii of the absorbing sets, which in turn are unrealistically large. This is a usual outcome of mathematical analysis of these type of models, it is more surprising that in \cref{teodetfunct} we could obtain the existence of asymptotically determining modes for physically realistic values of the parameters.
 %
\end{rem}

\appendix
\section{Regularity of the Laplacian in a rectangular domain}
\label{rectdom}
In this appendix we want to sketch some regularity results for the Laplacian on domains which are rectangles. 
The classic regularity in $H^2$ for the solution of the Dirichlet problem is a well-known result also for nonsmooth convex domains, see \cite[Theorem 3.2.1.2 and Theorem 4.3.1.4]{Grisv}.
In order to get more regularity of the solution of the Dirichlet problem on the rectangle an extra assumption on the non-homogeneous term $f$ is needed to obtain a similar result to \cite[Proposition 5.1]{LM}, which holds only for regular domains. This extra assumption on $f$ is stated in in the following theorem:
\begin{thm}[\cite{compcond}, Proposition 1]
    \label{teocompcond}
    For given $f\in H^{2k}$, the solution of the Dirichlet problem lies in $H^{2k+2}$ if and only if \[ \restr{C_jf}{\mathcal{C}}:=\sum_{i=1}^j(-1)^{i+1}\partial_x^{2j-2i}\partial_y^{2i-2}\restr{f}{\mathcal{C}}=0, \quad \text{for all } j=1,\dots, k,\]
    where $\mathcal{C}$ denotes the set consisting of the four corners of the rectangle $\Omega$.
\end{thm}
The above theorem can be used to show that $D(A^k)$ as defined in \eqref{specdec} is a subset of $H^{2k}$. First, we can observe that if $\restr{A^ju}{\partial\Omega}=0$ for all $j=0,\dots,k-1$ then $C_l \restr{A^{k-l}u}{\mathcal{C}}=0$ for all $l=1,\dots,k$: indeed, for example on the boundary on which $y$ is constant one has $\partial_x^{i}\partial_y^{2j}u=0$ for all $i,j\in\bbN$. Now, if $u\in D(A^k)$ then $A^{k-1}u\in D(A)$. Since $A^{k-1}u$ satisfies the compatibility condition $C_1A^{k-1}u=0$ on $\mathcal{C}$, applying Theorem~\ref{teocompcond} we get $A^{k-2}u\in H^4$. Then $A^{k-2}u$ also fulfils the compatibility conditions $C_1A^{k-2}u=C_2A^{k-2}u=0$ on $\mathcal{C}$, hence we get $A^{k-3}u\in H^6$. Iterating this procedure we find $u\in H^{2k}$. For the odd powers one can follow the same procedure and find $D(A^{k+\frac{1}{2}})\subset H^{2k+1}$.
This inclusion implies that we can write $D(A^k)$ as in \eqref{defDAalter}, and
the equivalence between the norms $\Vert \cdot \Vert_{H^k}$ and $|A^{k/2}\cdot|$ is a consequence of the Open Mapping Theorem applied to $\mathcal{I}: H^k\rightarrow D(A^{k/2})$.

\section*{Acknowledgments} The work presented here greatly benefited from fruitful discussions with a number of colleagues. In particular we are grateful to Stephane Vannitsem, Jochen Bröcker, Jonathan Demaeyer, and Valerio Lucarini. TK and GC are members of the INdAM-GNAMPA. TK gratefully acknowledge the partial support by PRIN2022-
PNRR–Project N.P20225SP98 ``Some mathematical approaches to climate change and its impacts''.
\bibliographystyle{plain}
\bibliography{references}

\noindent Federico Fornasaro$^1$\\
{\footnotesize
Dipartimento di Matematica\\
Università degli Studi di Roma La Sapienza \\
Email: \url{federico.fornasaro@uniroma1.it}\\
}

\noindent Tobias Kuna$^2$\\
{\footnotesize
Dipartimento di Ingegneria e Scienze dell'Informazione e Matematica \\
Universit\`a degli Studi Dell'Aquila \\
Email: \url{tobias.kuna@univaq.it}\\
}

\noindent Giulia Carigi$^3$\\
{\footnotesize
Department of Statistics\\
Indiana University Bloomington \\
Email: \url{gcarigi@iu.edu} \\
}

\end{document}